\documentclass[12pt]{amsart}  
\usepackage{amsmath,amsthm,amssymb,mathdots}
\usepackage{setspace,shuffle,enumitem,MnSymbol,stmaryrd}
\usepackage[all]{xy}
\usepackage{color}
\allowdisplaybreaks
\usepackage{here}
\newtheorem{theorem}{Theorem}[section]
\newtheorem{proposition}[theorem]{Proposition}
\newtheorem{lemma}[theorem]{Lemma}
\newtheorem{corollary}[theorem]{Corollary}

\theoremstyle{definition}
\newtheorem{definition}[theorem]{Definition}
\newtheorem{convention}[theorem]{Convention}

\theoremstyle{remark}

\newtheorem{remark}[theorem]{Remark}
\newtheorem{example}[theorem]{Example}

\numberwithin{equation}{section}

\newcommand{\R}{\mathbb R}
\newcommand{\Z}{{\mathbb Z}}
\newcommand{\N}{{\mathbb N}}
\newcommand{\C}{{\mathbb C}}
\newcommand{\Q}{{\mathbb Q}}
\newcommand{\gp}{{\Pi}}

\DeclareMathOperator{\wt}{wt}
\DeclareMathOperator{\dep}{dep}
\newcommand{\bk}{{\boldsymbol{k}}}
\newcommand{\bl}{{\boldsymbol{l}}}

\title{Symmetric multiple Eisenstein series}
\author{Takashi Hara, Kenji Sakugawa, Koji Tasaka}

\address[Takashi Hara]{College of Liberal Arts, Tsuda University, 2-1-1 Tsuda-machi, Kodaira City, Tokyo 187-8577, Japan}
\email{t-hara@tsuda.ac.jp}
\address[Kenji Sakugawa]{Faculty of Education, Shinshu University, 6-Ro Nishi-nagano, Nagano City, Nagano 380-8544, Japan}
\email{sakugawa\_{}kenji@shinshu-u.ac.jp}
\address[Koji Tasaka]{Faculty of Science and Engineering, Kindai University, 3-4-1 Kowakae, Higashiosaka City, Osaka 577-8502, Japan}
\email{tasaka@math.kindai.ac.jp}
\date{\today}
\begin{document}

\begin{abstract}
In this paper, we introduce the \emph{symmetric multiple Eisenstein series}, a variant of the multiple Eisenstein series. 
As a fundamental result, we show that they satisfy the linear shuffle relation.
As a case study, we investigate the vector space spanned by symmetric double Eisenstein series of weight $k$.
When $k$ is even, it coincides with the space spanned by modular forms of weight $k$ and the derivative of the Eisenstein series of weight $k-2$.
For $k$ odd, we prove that its dimension equals $\lfloor k/3\rfloor$.
We further provide an explicit correspondence between the linear shuffle relation and the Fay-shuffle relation satisfied by elliptic double zeta values, which may be of independent interest.
In connection with modular forms, we prove that every modular form can be expressed as a linear combination of symmetric \emph{triple} Eisenstein series.
This will serve as a first step toward understanding modular phenomena for symmetric multiple zeta values observed by Kaneko and Zagier.
\end{abstract}

\maketitle

\setcounter{tocdepth}{1}
\tableofcontents

\section{Introduction}

The purpose of the present paper is to introduce and investigate the \emph{symmetric multiple Eisenstein series} $G_{\bk}^{\shuffle,S}(\tau)$ for each tuple of positive integers $\bk\in \mathbb{N}^d$, where $\tau$ denotes an element in the complex upper half-plane. Throughout the article, the symbol $\mathbb{N}$ denotes the set of all {\em positive} integers, not including $0$.

\subsection{Main results}

The function $G^{\shuffle,S}_{\bk}(\tau)$ which we shall introduce is a variant of the multiple Eisenstein series $G_{k_1,\ldots,k_d}(\tau)$ which is defined for integers $k_1,\ldots,k_{d-1}\ge2$ and $k_d\ge3$ by the absolutely convergent series
\[ G_{k_1,\ldots,k_d}(\tau)=\sum_{\substack{0<\lambda_1< \cdots < \lambda_d\\ \lambda_1,\ldots,\lambda_d \in \Z\tau+\Z}}\frac{1}{ \lambda_1^{k_1}\cdots \lambda_d^{k_d}}. \]
Here, for lattice points $\lambda, \mu \in\Z\tau+\Z$, 
we write $\lambda <\mu$ if 
$\mu-\lambda$ is contained in the complex upper half-plane or on the positive real half-line.
The function $G_{\bk}(\tau)$, which is holomorphic on the complex upper half-plane, was first introduced by Gangl, Kaneko, and Zagier in \cite[Section 1]{GKZ} in the case $d=2$ (double Eisenstein series) in connection with the multiple zeta values $\zeta(\bk)$ (see \eqref{eq:MZV_def} for the definition).
For each $\bk=(k_1,\ldots,k_d)\in \mathbb{N}^d$, our function $G_{\bk}^{\shuffle,S}(\tau)$ is defined from the shuffle regularization $G^{\shuffle}_{\bk}(\tau)$ of $G_{\bk}(\tau)$, established in \cite[Definition 4.8]{BT}, as follows:
\begin{equation*}\label{eq:sym_mes_reg}
G^{\shuffle,S}_{\bk}(\tau)
:= \sum_{j=0}^d (-1)^{k_{j+1}+k_{j+2}+\cdots+k_d}
   G^{\shuffle}_{k_1,\ldots,k_j}(\tau)\,
   G^{\shuffle}_{k_d,\ldots,k_{j+1}}(\tau).
\end{equation*}
Note that, when $d=1$, the function $G^{\shuffle,S}_k(\tau)$ coincides with the classical Eisenstein series of weight $k$.

One of the motivations for studying the function $G^{\shuffle,S}_{\bk}(\tau)$ comes from the conjectural relationship between symmetric \emph{triple} zeta values (the case $d=3$) and modular forms for $\mathrm{SL}_2(\mathbb{Z})$ proposed by Kaneko and Zagier~\cite{KZ}; see Remark~\ref{rem:KZ} for details.
Here for each $\bk = (k_1,\ldots,k_d) \in \mathbb{N}^d$, the \emph{symmetric multiple zeta value} $\zeta^{\shuffle,S}(\bk)$ is defined by
\begin{equation}\label{eq:sym_mzv_reg}
\zeta^{\shuffle,S}(\bk)
= \sum_{j=0}^d (-1)^{k_{j+1}+\cdots+k_d}
  \zeta^\shuffle(k_1,\ldots,k_j)\,
  \zeta^\shuffle(k_d,\ldots,k_{j+1}),
\end{equation}
where $\zeta^{\shuffle}(\bk)$ is the shuffle regularization of $\zeta(\bk)$.
The number $\zeta^{\shuffle,S}(\bk)$ appears as the constant term of the $q$-series $G^{\shuffle,S}_{\bk}(\tau)$; recall that $G^{\shuffle}_{\bk}(\tau)$ is defined as a power series in $q=e^{2\pi i\tau}$ whose constant term is given by $\zeta^{\shuffle}(\bk)$.
Accordingly, symmetric multiple Eisenstein series have direct applications to the study of symmetric multiple zeta values.
Using this terminology, we aim to understand the relationship between symmetric multiple zeta values and modular forms via symmetric multiple Eisenstein series.
Here, we note a difference in terminology from that used in the literature; although symmetric multiple zeta values were originally introduced as elements of the quotient ring $\mathcal{Z}/\zeta(2)\mathcal{Z}$, where $\mathcal{Z}$ denotes the ring of multiple zeta values, we refer in this paper to the real number $\zeta^{\shuffle,S}(\bk)$ itself as the symmetric multiple zeta value.

\

As a first step toward clarifying the conjectural relationship above, we obtain the following result, which is the first main result of the present article.
For each index $\bk = (k_1,\ldots,k_d) \in \mathbb{N}^d$, let $\widetilde{G}_{\bk}^{\shuffle,S}(\tau)=G_{\bk}^{\shuffle,S}(\tau)/(2\pi i)^{k_1+\cdots+k_d}$.

\begin{theorem}\label{thm:modular_relations}
Every cusp form of weight $k$ for ${\rm SL}_2(\Z)$ with rational Fourier coefficients is presented as a $\Q$-linear combination of symmetric triple Eisenstein series $\widetilde{G}_{k_1,k_2,k_3}^{\shuffle,S}(\tau)$ for $k_1,k_2,k_3\in \N$ with $k=k_1+k_2+k_3$.
\end{theorem}

From Theorem~\ref{thm:modular_relations}, one can express any cusp form of weight~$k$ in terms of symmetric triple Eisenstein series. Taking the constant term of its $q$-expansion then yields a $\mathbb{Q}$-linear relation among symmetric triple zeta values.
For example, we numerically obtain
\begin{align*}
 \frac{67}{64800} \Delta=&-3421404 \widetilde{G}^{\shuffle,S}_{1,1,10}-1140468 \widetilde{G}^{\shuffle,S}_{1,2,9}-885388 \widetilde{G}^{\shuffle,S}_{1,3,8}-789612 \widetilde{G}^{\shuffle,S}_{1,4,7}\\
&-673924 \widetilde{G}^{\shuffle,S}_{1,5,6}-595458 \widetilde{G}^{\shuffle,S}_{1,6,5}-502768    \widetilde{G}^{\shuffle,S}_{1,7,4}-332318 \widetilde{G}^{\shuffle,S}_{1,8,3}\\
&+63770 \widetilde{G}^{\shuffle,S}_{2,2,8}+47888 \widetilde{G}^{\shuffle,S}_{2,3,7}+46253\widetilde{G}^{\shuffle,S}_{2,4,6}+26007 \widetilde{G}^{\shuffle,S}_{2,5,5},
\end{align*}
where we omit the variable $\tau$, and
$\Delta=\Delta(\tau)=q\prod_{n=1}^\infty (1-q^n)^{24}$ with $q=e^{2\pi \tau}$ denotes the unique normalized cusp form (modular discriminant) of weight~$12$ for ${\rm SL}_2(\Z)$.
The expression is not unique, as it depends on the choice of a basis for the space spanned by symmetric triple Eisenstein series.

For comparison, we briefly review earlier results on double Eisenstein series.
Recall that Theorem~\ref{thm:modular_relations} remains valid if symmetric triple
Eisenstein series are replaced by double Eisenstein series.
More precisely, for $r,s \ge 1$, let $
\widetilde{G}_{r,s}^\shuffle(\tau)
:= \frac{G_{r,s}^\shuffle(\tau)}{(2\pi i)^{r+s}}$.
Then, for $k \ge 4$, the $\mathbb{Q}$-vector space
\[
\mathcal{DE}_k
:= \left\langle
\widetilde{G}^\shuffle_{r,s}(\tau)
\ \middle| \ r+s=k,\ r\ge1,\ s\ge2
\right\rangle_{\mathbb{Q}}
\]
of double Eisenstein series of weight $k$ contains all modular forms of weight $k$
for $\mathrm{SL}_2(\mathbb{Z})$ with rational Fourier coefficients; see \cite{GKZ} for details.
This again implies that there exist $\mathbb{Q}$-linear relations among double zeta values
arising from cusp forms.
Gangl, Kaneko, and Zagier obtained in \cite[Theorem 3]{GKZ} an explicit description of such relations using even period polynomials, which are in one-to-one correspondence with modular forms.
Moreover, as a refinement of their result, explicit formulas expressing normalized cuspidal eigenforms for ${\rm SL}_2(\Z)$
as linear combinations of a certain basis of the space $\mathcal{DE}_k$ were established in \cite[Theorem 1]{Tasaka2}.
For further developments in this direction, we refer the reader to the survey article~\cite{Tasaka3}.

Thanks to these developments, the relations among double zeta values arising from cusp forms
are now well understood.
In contrast to the case of double Eisenstein series, for symmetric triple Eisenstein series,
neither a suitable choice of a basis nor an explicit formula for a given cusp form
has been established so far.

\

We now turn to the second main result of the present article, concerning the determination of the space of symmetric double Eisenstein series.
This result leads to their connections with modular forms and elliptic double zeta values.
For a positive integer $k\ge3$, let us consider
\begin{equation}\label{eq:DE_k^S}
\mathcal{DE}_k^S:=\left\langle \widetilde{G}^{\shuffle,S}_{r,s}(\tau)\ \middle| \ r+s=k,\ r,s\ge1\right\rangle_\Q.
\end{equation}
Note that, for $k\ge4$, the inclusion $\mathcal{DE}_k^S\subset \mathcal{DE}_k$ is obtained as a consequence of Propositions \ref{prop:Gsym_dep2_even} and \ref{prop:Gsym_dep2_odd} and the sum formulas \eqref{eq:sum_formula} and \eqref{eq:sum_formula_2}. 
For $k\ge0$, $M_k^\Q ({\rm SL}_2(\Z))$ denotes the $\Q$-vector space spanned by modular forms of weight $k$ for ${\rm SL}_2(\Z)$ with rational Fourier coefficients.

\begin{theorem}\label{thm:main}
\begin{enumerate}[leftmargin=23pt, label={\rm (\roman*)}]
\item For $k\ge6$ even, we have 
\[\mathcal{DE}_k^S= M_k^\Q ({\rm SL}_2(\Z))\oplus
\Q  \widetilde{G}'_{k-2}(\tau),\]
where $\widetilde{G}'_{k-2}(\tau)$ is the derivative
$\frac{1}{2\pi i}\frac{d}{d\tau}\widetilde{G}_{k-2}^\shuffle(\tau)$ of
$\widetilde{G}_{k}^\shuffle(\tau)
= \frac{G_{k}^\shuffle(\tau)}{(2\pi i)^k}$.
Thus, its dimension is given by
\[\dim_\Q \mathcal{DE}_k^S =  \left\lfloor \frac{k+4}{4}\right\rfloor - \left\lfloor \frac{k-2}{6}\right\rfloor.\]
\item For $k\ge3$ odd, the set $\big\{\widetilde{G}_{j,k-j}^{\shuffle,S}(\tau) \mid  j=1,2,\ldots, \left\lfloor \frac{k}{3} \right\rfloor \big\}$ forms a basis of $\mathcal{DE}_k^S$.
In particular, we obtain
\[\dim_\Q \mathcal{DE}_k^S =\left\lfloor \frac{k}{3}\right\rfloor.\] 
\end{enumerate}
\end{theorem}

For the proof of the inequality
$\dim_{\Q} \mathcal{DE}_k^S \le \left\lfloor \frac{k}{3} \right\rfloor$ in Theorem~\ref{thm:main} (ii), we introduce the \emph{linear shuffle space} ${\rm LSh}_w^{(2)}$, a dual framework for the linear shuffle relation satisfied by symmetric multiple Eisenstein series.
It is defined as a $\mathbb{Q}$-vector subspace of the space of homogeneous polynomials of degree~$w$ (see Definition~\ref{def:lin_sh_space}).
By duality, we obtain an inequality $\dim_{\Q} \mathcal{DE}_k^S \le \dim_{\Q} {\rm LSh}_{k-2}^{(2)}$ (Proposition~\ref{prop:bound_E}).
The desired bound then follows from the dimension formula $\dim_{\Q} {\rm LSh}_{k-2}^{(2)} = \left\lfloor \frac{k}{3} \right\rfloor$, which is established due to representation theory of the symmetric group $\mathfrak{S}_3$ (Theorem~\ref{thm:dim_LSh2}).

The linear shuffle space ${\rm LSh}_w^{(2)}$ itself is of independent interest.
We also establish its connections with the space $W_w^{\rm od}$ of \emph{odd period polynomials}, as well as with the polynomial subspace ${\rm FSh}_w^{\mathrm{pol}}$ of the \emph{Fay-shuffle space}, introduced by Matthes~\cite{Matthes} as a dual setup of the Fay-shuffle relation satisfied by elliptic double zeta values.

\begin{theorem}[=Theprem~\ref{thm:lin_sh_vs_fay}] \label{thm:main2}
\begin{enumerate}[leftmargin=23pt, label={\rm (\roman*)}]
\item For a positive even integer $w$, the space $W_w^{\rm od}$ of odd period polynomials is contained in ${\rm LSh}_w^{(2)}$.
\item For a positive odd integer $w$, there is an explicit isomorphism between the polynomial part of the Fay-shuffle space ${\rm FSh}_w^\mathrm{pol}$ and ${\rm LSh}_w^{(2)}$.
\end{enumerate}
\end{theorem}

\subsection{Organization of the article}

The contents of this paper are organized as follows. 
In Section~2, we first recall some basic facts about multiple zeta values and fix our notation. We then review the results of Bachmann and the third-named author from \cite{BT}, including the Fourier expansion of the multiple Eisenstein series, its relation with the Goncharov coproduct and the regularized multiple Eisenstein series $G_{\bk}^\shuffle(\tau)$.
Furthermore, an alternative formulation of $G_{\bk}^\shuffle(\tau)$ based on the Ihara group law and a possible rational structure of the space of multiple Eisenstein series will be discussed. The detailed study of the relation between the Goncharov coproduct and the Ihara group law (Theorem~\ref{thm:Gon_Ih}) will be developed in Appendix~\ref{appendix:Ihara-Goncharov}.
In Section~3, we review the series representation and the linear shuffle relation of symmetric multiple zeta values, and carry out an analogous analysis for symmetric multiple Eisenstein series. We also present a numerical observation on the $\Q$-vector space spanned by symmetric multiple Eisenstein series.
Sections~4 and 5 are devoted to the definition and study of the linear shuffle space. We mainly consider the case when the depth is $2$. In this case, connections with the theory of modular forms and elliptic double zeta values can be observed.
In Section~6, we prove the main results, namely Theorems~\ref{thm:modular_relations},
\ref{thm:main}, and~\ref{thm:main2}.
A part of the proof, namely the rank estimation of a certain integer matrix used
to establish the lower bound in Theorem~\ref{thm:main} (ii), is postponed to Appendix~\ref{appendix:rank_estimate}.

\section{Multiple Eisenstein series}
In this section, following \cite{BT}, we overview the theory of multiple Eisenstein series and their shuffle regularization which we shall use  in Sections \ref{ssc:MES} and \ref{ssc:Shuffle_regularization}. Section \ref{subsec:alg_setup} is a preliminary subsection summarizing algebraic setups on multiple zeta values.
On this occasion, we also mention some updates on this theory. 
Specifically, these include the description of the shuffle regularization based on the noncommutative generating series and the Ihara group law in Section \ref{ssc:Ihara}, and the result concerning a rational structure of the $\Q$-vector space of multiple Eisenstein series in Section \ref{ssc: comparison_LR}.

\subsection{Algebraic setups on multiple zeta values}\label{subsec:alg_setup}
We set up basic notation on multiple zeta values.
Throughout the paper, a tuple of positive integers is said to be an \emph{index}. 
For an index $\bk=(k_1,\ldots,k_d)$, $\wt(\bk)=k_1+\cdots+k_d$ and $\dep(\bk)=d$ are called the \emph{weight} and the \emph{depth} of $\bk$, respectively.
If $\bk=\varnothing$ is the empty index, we set $\wt(\varnothing)=\dep(\varnothing)=0$. Following this convention, we regard $\mathbb{N}^0$ as a symbol denoting the singleton $\{\varnothing\}$.

For an index $\bk=(k_1,\ldots,k_d)$ with $k_d\ge2$, we define the \emph{multiple zeta value} $\zeta(\bk)$ by the absolutely convergent series
\begin{equation}\label{eq:MZV_def} 
\zeta(\bk):=\sum_{0<n_1<\cdots<n_d}\frac{1}{n_1^{k_1}\cdots n_d^{k_d}}.
\end{equation}
We set $\zeta(\varnothing)=1$.
For a non-negative integer $k$, let $\mathcal{Z}_k$ be the $\Q$-vector space spanned by all multiple zeta values of weight $k$ and $\mathcal{Z}:=\bigoplus_{k\ge0}\mathcal{Z}_k$ the formal direct sum of $\mathcal{Z}_k$.
The space $\mathcal{Z}$ forms a graded $\Q$-algebra; in other words, the product of two multiple zeta values of weight $k$ and $l$ is written as a $\Q$-linear combination of multiple zeta values of weight $k+l$.

Let us recall two algebraic structures of $\mathcal{Z}$. Refer, for example, to \cite[Sections 1, 2, 3]{IKZ06} for details.
Let $\mathfrak{H}=\Q\langle e_0,e_1\rangle$ be the non-commutative polynomial algebra in two indeterminates $e_0$ and $e_1$ over $\Q$. We use the symbol $\mathfrak{H}_\shuffle$ when we regard $\mathfrak{H}$ as the \emph{shuffle algebra} equipped with the \emph{shuffle product} $\shuffle \, \colon \mathfrak{H}\times \mathfrak{H}\to \mathfrak{H}$.
The $\Q$-bilinear map $\shuffle$ is defined inductively by $e_aw\shuffle e_bw' =e_a(w\shuffle e_bw')+e_b(e_aw\shuffle w')$ for $a,b\in \{0,1\}$ and any words $w,w'$ in $e_0$ and $e_1$ with the initial condition $1\shuffle w=w=w\shuffle 1$.
Set $\mathfrak{H}^1:=\Q+e_1\mathfrak{H}$, and $\mathfrak{H}^0:=\Q+e_1\mathfrak{H}e_0$.
They form subalgebras $\mathfrak{H}^1_\shuffle$ and $\mathfrak{H}^0_\shuffle$ of $\mathfrak{H}_\shuffle$ satisfying  $\mathfrak{H}^0_\shuffle \subset \mathfrak{H}^1_\shuffle \subset \mathfrak{H}_\shuffle$.
For an integer $k\ge1$, we write $e_k=e_1e_0^{k-1}$ and, for an index $\bk=(k_1,\ldots,k_d)$, we put
\[ e_{\bk} :=e_{k_1}\cdots e_{k_d}=e_1e_0^{k_1-1}\cdots e_1e_0^{k_d-1}.\]
The set $\{e_\bk \mid \bk \in \N^d , d\ge0\}$ forms a basis of $\mathfrak{H}^1$, where $e_\varnothing$ denotes $1$ by convention. 
The subspace $\mathfrak{H}^0$ is generated by 1 and all words that start with $e_1$ and end with $e_0$.

Due to the iterated integral representation of multiple zeta values, the $\Q$-linear map 
\[\mathsf{Z}\colon \mathfrak{H}^0 \longrightarrow \mathcal{Z};\quad e_{\bk}\longmapsto \zeta(\bk)\]
becomes a $\shuffle$-algebra homomorphism $\mathsf{Z}\colon \mathfrak{H}^0_\shuffle \longrightarrow \mathcal{Z}$.
Since $\mathfrak{H}_\shuffle\cong \mathfrak{H}_\shuffle^1[e_0]\cong \mathfrak{H}_\shuffle^0[e_0,e_1]$ hold as shuffle algebras, there is a unique $\shuffle$-algebra homomorphism 
\begin{align} \label{eq:Z_shuffle}
    \mathsf{Z}^\shuffle \colon\mathfrak{H}_\shuffle \longrightarrow \mathcal{Z} 
\end{align} 
such that $\mathsf{Z}^\shuffle(e_1)=\mathsf{Z}^\shuffle(e_0)=0$ and $\mathsf{Z}^\shuffle\big|_{\mathfrak{H}^0}=\mathsf{Z}$.
For each index $\bk\in \N^d$, we write $\zeta^\shuffle (\bk)$ for $\mathsf{Z}^\shuffle (e_\bk)$ and call it the \emph{shuffle regularization} of $\zeta(\bk)$.
Note that we have $\zeta^\shuffle (\bk)=\zeta(\bk)$ if the last component of $\bk$ is strictly greater than 1.

Note that the isomorphism $\mathfrak{H}_\shuffle\cong \mathfrak{H}_\shuffle^1 [e_0]$ is obtained from the fact that every word $w\in \mathfrak{H}$ is uniquely written as $w=\sum w_j \shuffle e_0^{\shuffle j}=w_0+w_1\shuffle e_0+w_2\shuffle e_0\shuffle e_0+\cdots$ with $w_j\in \mathfrak{H}^1$.
Taking the constant term $w_0$ of $w$, we get a projection
\begin{equation}\label{eq:reg_0}
{\rm reg}_0 : \mathfrak{H}_\shuffle\longrightarrow \mathfrak{H}_\shuffle^1;\quad w\longmapsto w_0.
\end{equation} 
This is the algebra homomorphism such that ${\rm reg}_0 (e_0)=0$.
For integers $k_1,\ldots,k_d\ge1$ and $n\geq 0$, the following formula is known to hold:
\[{\rm reg}_0(e_0^ne_{k_1,\ldots,k_d})=(-1)^n \sum_{\substack{l_1+\cdots+l_d=n\\ l_1,\ldots,l_d\ge0}} \prod_{j=1}^d \binom{k_j+l_j-1}{l_j}e_{k_1+l_1,\ldots,k_d+l_d}. \]

For $e_{k},e_{\ell} \in \mathfrak{H}^1$ with $k,\ell\in \N$ and any words $w,w'\in\mathfrak{H}^1$ in $e_0$ and $e_1$, we define  the \emph{stuffle product} $\ast$ on $\mathfrak{H}^1$ inductively by
\begin{align*}
 &e_{k}  w \ast e_{\ell}  w' = e_{k} ( w\ast e_{\ell} w') + e_{\ell}(e_{k} w\ast w') + e_{k+\ell} (w \ast  w'),
\end{align*}
with initial condition $w\ast 1=1\ast w=w$. Then $*$ is extended $\mathbb{Q}$-bilinearly to the whole $\mathfrak{H}^1$. 
The stuffle product corresponds to the operation of rewriting the product of two multiple sums 
into a sum of multiple sums.
For instance, we have
\begin{align*}
\zeta(k_1)\zeta(k_2) &=\sum_{0<n_1,n_2}\frac{1}{n_1^{k_1}n_2^{k_2}} =\left( \sum_{0<n_1<n_2}+ \sum_{0<n_1<n_2}+ \sum_{0<n_1<n_2}\right)\frac{1}{n_1^{k_1}n_2^{k_2}}\\
&=\zeta(k_1,k_2)+\zeta(k_2,k_1)+\zeta(k_1+k_2) = \mathsf{Z}(e_{k_1}\ast e_{k_2})
\end{align*}
if both $k_1$ and $k_2$ are greater than or equal to $2$ (that is, when all the series converge absolutely).
This technique is available for multiple sums defined on a totally ordered set.
Later we use this for our objects, namely, the (symmetric) multiple Eisenstein series.
We write $\mathfrak{H}^1_\ast$ for $\mathfrak{H}^1$ when we regard it as a commutative $\Q$-algebra with respect to the stuffle product $\ast$. 
The $\mathfrak{H}^0$ forms a subalgebra of $\mathfrak{H}^1_\ast$, which is also denoted by $\mathfrak{H}^0_\ast$, and the map $\mathsf{Z}\colon \mathfrak{H}^0 \rightarrow \mathcal{Z}$ appearing before also becomes a homomorphism of $\mathbb{Q}$-algebras $\mathsf{Z}\colon \mathfrak{H}^0_\ast \rightarrow \mathcal{Z}$.

By combining the above two algebraic structures, we obtain the (finite) double shuffle relation: for any words $w, w' \in \mathfrak{H}^0$ in $e_0$ and $e_1$, we have
\[
\mathsf{Z}(w \ast w' - w \shuffle w') = \mathsf{Z}(w)\mathsf{Z}(w')-\mathsf{Z}(w)\mathsf{Z}(w')= 0.
\]

Hereafter, for a $\Q$-algebra $R$, a family of elements $A_{\bk}\in R$ indexed by all $\bk\in \mathbb{N}^d$ ($d\geq 0$) is said to satisfy the \emph{shuffle} (resp.~\emph{stuffle}) \emph{relation} if the $\Q$-linear map $\mathfrak{H}_\shuffle^1 \rightarrow R$ (resp.~$\mathfrak{H}_\ast^1 \rightarrow R$) defined by $e_\bk\mapsto A_\bk$ is an algebra homomorphism. 

\subsection{Multiple Eisenstein series and its Fourier expansion}
\label{ssc:MES}

We review in this subsection the definition of the multiple Eisenstein series and its Fourier expansion.

For a point $\tau$ of the complex upper half-plane, let us consider the lattice $\Z\tau+\Z$ of $\C$ generated by $\tau$ and $1$, and two subsets  
\begin{equation}\label{eq:P}
P_0:=\{n\in\Z \mid n>0\} \quad \mbox{and} \quad P_1:=\{m\tau+n\in \Z\tau+\Z\mid m>0\}.\end{equation}
of $\Z\tau+\Z$. A lattice point $\lambda\in \Z\tau+\Z$ is said to be {\em positive} if $\lambda$ is contained in the union $P_0\cup P_1$. We write $\lambda>0$ for positive $\lambda$.
We define an order $<$ on $\mathbb{Z}\tau+\mathbb{Z}$ by setting $\lambda<\mu$ for every  $\lambda,\mu \in\Z\tau+\Z$ satisfying $\mu-\lambda>0$.

\begin{definition}\label{def:MES}
For positive integers $k_1,\ldots,k_d\ge2$, we define the {\em multiple Eisenstein series} $G_{k_1,\cdots,k_d}(\tau)$ as the following conditionally convergent series:
\begin{equation*}
G_{k_1,\cdots,k_d}(\tau) :=  \lim_{M\rightarrow\infty}\lim_{N\rightarrow \infty} \sum_{\substack{0<\lambda_1< \cdots < \lambda_d\\ \lambda_1,\ldots,\lambda_d \in \Z_M\tau+\Z_N}} \lambda_1^{-k_1}\cdots \lambda_d^{-k_d} ,
\end{equation*}
where $\Z_M$ for $M\ge0$ is defined as a finite subset $\{m\in \Z \mid -M\le m\le M\}$ of $\Z$.
\end{definition}

Now let $\bk=(k_1,\ldots,k_d)$ be an index.
In order to describe the Fourier expansion of the multiple Eisenstein series $G_{\bk}(\tau)$, 
we define a holomorphic function $g_{\bk}(\tau)$ on the complex upper half-plane by the $q$-expansion
\begin{equation*}\label{eq2_4}
g_{\bk} (\tau) := \frac{(-2\pi i)^{\wt(\bk)}}{(k_1-1)!\cdots (k_d-1)!}\sum_{\substack{\ell_1,\ldots,\ell_d\in\N\\m_1,\ldots,m_d\in\N\\0<m_1<\cdots<m_d}} \ell_1^{k_1-1}\cdots \ell_d^{k_d-1}q^{\ell_1m_1+\cdots+\ell_dm_d}\qquad (q=e^{2\pi i\tau}).
\end{equation*}
For example, we have $g_k(\tau)=\frac{(-2\pi i)^k}{(k-1)!} \sum_{n>0} \big(\sum_{0<d\mid n}d^{k-1}\big)q^n$ for $k\ge1$, whose coefficients involve the divisor sums. 
The coefficients of $g_{\bk} (\tau)/(-2\pi i)^{\wt(\bk)}$ are in general called the \emph{multiple divisor sum}; see \cite[(1.1)]{BK} for details.

The multiple Eisenstein series is written as a $\mathcal{Z}$-linear combination of $g_\bk (\tau)$'s.
For example, we have
\begin{equation}\label{eq:example_fourier}
\begin{aligned}
G_{3,3}(\tau)&=\zeta(3,3)-6\zeta(4)g_2(\tau)+\zeta(3)g_3(\tau)+g_{3,3}(\tau),\\
 G_{2,2,3}(\tau)&= \zeta(2,2,3)+(4\zeta(2)\zeta(3)+3\zeta(2,3)+2\zeta(2,3))g_2(\tau)\\
 &\quad +(\zeta(2)^2+2\zeta(2,2))g_3(\tau)+3\zeta(3)g_{2,2}(\tau)+4\zeta(2) g_{2,3}(\tau)+g_{2,2,3}(\tau).
 \end{aligned}
 \end{equation}
In general, one can compute the Fourier expansion of the multiple Eisenstein series by using the following expression in terms of $g_{\bk}(\tau)$. For details, see \cite[Section 2]{BT}.

\begin{proposition}\label{prop:fourier}
Let $\bk=(k_1,\ldots,k_d)$ be an index of weight $k$ with $k_1,\ldots,k_d\ge2$.
Then, for each index $\bl$, there exists $a_{\bk,\bl} \in \mathcal{Z}_{k-\wt(\bl)}$ such that 
\[ G_{\bk}(\tau) = \zeta(\bk)+\sum_{j=1}^{d-1} \sum_{\substack{\bl\in \N^{j}\\ 1< \wt(\bl)<k}} a_{\bk,\bl}\,g_{\bl}(\tau)+g_\bk(\tau).\]
\end{proposition}

Since it plays an important role in the shuffle regularization, we here briefly recall the relationship between the Fourier expansion of multiple Eisenstein series and the {\em Goncharov coproduct} on formal iterated integrals $\mathbb{I}(a_0;a_1,\ldots,a_n;a_{n+1})$.
We use the notation $\Delta_G$ for the coproduct on the ring of formal iterated integrals; see \eqref{eq:GonCop} or \cite[(15)]{BT} for the definition.
See also \cite[Section 2]{G} for the original definition, in which the tensor factors are interchanged.
For $a_1,\ldots,a_n\in\{0,1\}$, using the relation $\mathbb{I}(1;a_n,\ldots,a_1;0)=(-1)^n\mathbb{I}(0;a_1,\ldots,a_n;1)$ (refer to \cite[Proposition 2.1]{G}) and identifying $\mathbb{I}(0;a_1,\ldots,a_n;1)$ with $e_{a_1}\cdots e_{a_n}$, one can define the coproduct $\Delta_G$ as a homomorphism of $\mathbb{Q}$-algebras 
\begin{equation}\label{eq:Delta_G}
\Delta_G \colon \mathfrak{H}_\shuffle \longrightarrow
\mathfrak{H}_\shuffle \otimes_{\mathbb{Q}} \mathfrak{H}_\shuffle.
\end{equation}
With respect to $\Delta_G$, the shuffle algebra $\mathfrak{H}_\shuffle$ becomes a Hopf algebra over $\Q$.
Since $\Delta_G(e_0)=1\otimes e_0+e_0\otimes 1$ holds, 
the coproduct $\Delta_G$ induces a Hopf algebra structure on the quotient  $\mathfrak{H}_\shuffle/(e_0)$, where $(e_0)$ denotes the ideal of $\mathfrak{H}_\shuffle$ generated by $e_0$.
The isomorphism $ \mathfrak{H}_\shuffle^1[e_0]\cong \mathfrak{H}_\shuffle$ gives rise to the algebra isomorphism $\mathfrak{H}_\shuffle^1\rightarrow \mathfrak{H}_\shuffle/(e_0);\ e_{\bk}\mapsto e_{\bk}+(e_0)$.
Its inverse is given by $\mathfrak{H}_\shuffle/(e_0)\rightarrow \mathfrak{H}_\shuffle^1; w+(e_0)\mapsto {\rm reg}_0(w)$, where ${\rm reg}_0:\mathfrak{H}_\shuffle\rightarrow \mathfrak{H}_\shuffle^1$ is defined in \eqref{eq:reg_0}.
Thus, the coproduct
\[\Delta_G^1:=({\rm reg}_0\otimes {\rm reg}_0)\circ \Delta_G\big|_{\mathfrak{H}_\shuffle^1}\colon \mathfrak{H}_\shuffle^1\longrightarrow \mathfrak{H}_\shuffle^1\otimes_\Q \mathfrak{H}_\shuffle^1\] 
induces a Hopf algebra structure on $\mathfrak{H}_\shuffle^1$.
As explained in \cite[Section 3]{BT}, it has the following expression in terms of the basis $\{e_\bk\mid \bk\in \N^d, d\ge0\}$ of $\mathfrak{H}^1_\shuffle$:
for each $\bk\in \N^d$ of weight $k$, we have
\begin{equation}\label{eq:Delta_G^1}
\Delta_G^1(e_\bk) = e_\bk \otimes 1 +\sum_{j=1}^{d-1} \sum_{\substack{\bl\in \N^{j}\\ 0< \wt(\bl)<k}} f_{\bk,\bl}\otimes e_{\bl}+1\otimes e_\bk
\end{equation}
for some $f_{\bk,\bl}\in \mathfrak{H}^1$ of weight $k-\wt(\bl)$, not necessarily a word.
For example, we have
\begin{equation*}\label{eq:example_Goncharov}
\begin{aligned}
\Delta_G^1(e_{3,3})&=e_{3,3}\otimes1-6e_4\otimes e_2+e_3\otimes e_3+1\otimes e_{3,3},\\
\Delta_G^1(e_{2,2,3})&= e_{2,2,3}\otimes 1+(4e_{2}\shuffle e_3+3 e_{2,3}+2e_{2,3})\otimes e_2\\
 &\quad +(e_2\shuffle e_2+2e_{2,2})\otimes e_3+3e_3\otimes e_{2,2}+4e_2\otimes e_{2,3}+1\otimes e_{2,2,3},
 \end{aligned}
 \end{equation*}
 which correspond to the Fourier expansions given in \eqref{eq:example_fourier}.
 This correspondence was first observed by Masanobu Kaneko and later studied by Stephanie Belcher.
It was completed in \cite[Theorem 1.1]{BT} as follows.
 
 \begin{theorem}\label{thm:Fourier_v.s._Goncharov} 
 Let $\mathcal{O}$ denote the ring of holomorphic functions on the complex upper half-plane and define the $\Q$-linear map $\mathsf{g}\colon\mathfrak{H}^1\to \mathcal{O}$ by $\mathsf{g}(e_\bk)=g_{\bk}(\tau)$.
 Then, for an index $\bk$ whose entries are all strictly greater than $1$, we have
 \begin{align} \label{eq:Gk_vs_coproduct}
 m\circ (\mathsf{Z} \otimes \mathsf{g} )\circ \Delta_G^1(e_\bk) = G_\bk (\tau),
 \end{align}
 where $m\colon \C\otimes_{\mathbb{Q}}\mathcal{O}\rightarrow\mathcal{O}$ denotes the natural multiplication map.
\end{theorem}

\subsection{Shuffle regularization of multiple Eisenstein series}
\label{ssc:Shuffle_regularization}

Since $\Delta_G^1$ is an algebra homomorphism, by replacing the maps $\mathsf{Z}$ and $\mathsf{g}$ in \eqref{eq:Gk_vs_coproduct} with appropriate {\em algebra homomorphisms} $\mathfrak{H}_\shuffle^1\rightarrow \mathcal{Z}$ and $\mathfrak{H}_\shuffle^1\rightarrow \mathcal{O}$, we can define a \emph{shuffle regularization} $G^\shuffle_\bk(\tau)$ of the multiple Eisenstein series $G_{\bk}(\tau)$.
Concerning $\mathsf{Z}$, one only has to replace it with the shuffle-regularized homomorphism $\mathsf{Z}^{\shuffle}\vert_{\mathfrak{H}^1_\shuffle}\colon \mathfrak{H}^1_\shuffle \rightarrow \mathcal{Z}$ defined as in \eqref{eq:Z_shuffle}.
On the other hand, since the function $g_\bk(\tau)$ does not satisfy the shuffle relation with respect to $\bk$, we need to introduce the shuffle regularization of $\mathsf{g}$, which is established in \cite[(25)]{BT}.

Let us consider a formal power series
\[ H^{(d)}\tbinom{k_1,\ldots,k_d}{x_1,\ldots,x_d}:=\sum_{0<m_1<\cdots<m_d} \prod_{j=1}^d e^{m_j x_j} \left( \frac{q^{m_j}}{1-q^{m_j}}\right)^{k_j}\in \Q\llbracket x_1,\ldots,x_d, q\rrbracket \]
for positive integers $k_1,\ldots,k_d$. This series serves as a generating function of $g_\bk(\tau)$'s when we set $q$ to be the holomorphic function $e^{2\pi i\tau}$ on the upper half-plane.
Indeed, we have
\[ \sum_{k_1,\ldots,k_d\ge1} \tilde{g}_{k_1,\ldots,k_d}(\tau) x_1^{k_1-1}\cdots x_d^{k_d-1}= H^{(d)}\tbinom{1,\ldots,1,1}{x_d-x_{d-1},\ldots,x_2-x_1,x_1}\big|_{q=e^{2\pi i\tau}} \]
where $\tilde{g}_{\bk}(\tau)$ denotes the normalized function $g_{\bk}(\tau)/(-2\pi i)^{\wt(\bk)}$ for each index $\bk$. We here remark that the series $H^{(d)}\tbinom{k_1,\ldots,k_d}{x_1,\ldots,x_d}$ satisfies the stuffle relation. 
For example, we have $H^{(1)}\tbinom{k_1}{x_1}H^{(1)}\tbinom{k_2}{x_2}=H^{(2)}\tbinom{k_1,k_2}{x_1,x_2}+H^{(2)}\tbinom{k_2,k_1}{x_2,x_1}+H^{(1)}\tbinom{k_1+k_2}{x_1+x_2}$.
Taking a weighted-average of $H^{(r)}\tbinom{k_1,\ldots,k_r}{x_1,\ldots,x_r}$'s, we define another formal power series $h^{(d)}(x_1,\ldots,x_d)$ by
\[ h^{(d)}(x_1,\ldots,x_d) := \sum_{r=1}^d \sum_{(j_1,\ldots,j_r)} \frac{1}{j_1!\cdots j_r !} H^{(r)}\tbinom{j_1,\ldots,j_r}{x_{j_1}',\ldots,x_{j_r}'},\]
where the inner sum runs over all tuples $(j_1,\ldots,j_r)$ of $r$ positive integers satisfying $j_1+j_2+\cdots+j_r=d$, and we set $x_{j_1}'=x_1+\cdots+x_{j_1},x_{j_2}'=x_{j_1+1}+\cdots+x_{j_1+j_2},\ldots,$ and $x_{j_r}'=x_{j_1+\cdots+j_{r-1}+1}+\cdots+x_d$. 
The explicit descriptions of $h^{(d)}(x_1,\ldots,x_d)$ for $d=1$ and $2$ are given in Example \ref{ex:g_shuffle}.
This averaging procedure, due to Hoffman \cite[Theorem 2.5]{H1}, transforms objects satisfying the stuffle relation into those satisfying the `index-shuffle' relation.
For instance, we have $h^{(1)}(x_1)h^{(1)}(x_2)=h^{(2)}(x_1,x_2)+h^{(2)}(x_2,x_1)$ which follows by substituting the defining expression of $h^{(d)}$ into the stuffle relation for $H^{(d)}$.
See \cite[Section 4.1]{BT} for details.

We are now ready to introduce the shuffle regularization of $g_{\bk}(\tau)$. We first introduce $\tilde{g}^\shuffle_{\bk}$ as a $q$-series appearing in the Taylor expansion of $h^{(d)}$.

\begin{definition}
For each index $\bk=(k_1,\ldots,k_d)$, we define the $q$-series $\tilde{g}_{\bk}^\shuffle \in \Q\llbracket q\rrbracket $ by
\begin{equation*}\label{def:g^sh} 
\sum_{k_1,\ldots,k_d\ge1} \tilde{g}_{k_1,\ldots,k_d}^\shuffle  \,x_1^{k_1-1}\cdots x_d^{k_d-1}:= h^{(d)}(x_d-x_{d-1},\ldots,x_2-x_1,x_1).
\end{equation*}
\end{definition}

Letting $q=e^{2\pi i\tau}$, we can view the $q$-series $ \tilde{g}_{\bk}^\shuffle$ as a holomorphic function on the upper half-plane, which is denoted as $ \tilde{g}_{\bk}^\shuffle (\tau):= \tilde{g}_{\bk}^\shuffle\big|_{q=e^{2\pi i\tau}}$.

\begin{example} \label{ex:g_shuffle}
When $d=1$, the series $h^{(1)}(x)$ coincides with $H^{(1)}\tbinom{1}{x}$ by definition, and thus we obtain
\begin{equation*}\label{eq;g_k-q-exp}
\tilde{g}^\shuffle_k(\tau)=\tilde{g}_k(\tau) = \frac{(-1)^k}{(k-1)!} \sum_{n\ge1} \left(\sum_{d\mid n} d^{k-1}\right) q^n \quad \text{for } k\ge1.
\end{equation*}
For the case $d=2$, we have 
\begin{align*}
h^{(2)}(x_2-x_1,x_1)&=H^{(2)}\tbinom{1,1}{x_2-x_1,x_1}+\frac12 H^{(1)}\tbinom{2}{x_2} \\
&= \sum_{r,s\ge1} \tilde{g}_{r,s}\,x_1^{r-1}x_2^{s-1}+\frac12 \sum_{m\ge1} e^{mx_2}\left(\frac{q^m}{1-q^m}\right)^2.
\end{align*} 
Hence, for positive integers $r,s\ge1$, we obtain
\[ \tilde{g}_{r,s}^\shuffle (\tau)= \tilde{g}_{r,s}(\tau) + \frac{\delta_{r,1}}{2(r+s-1)!} \sum_{\ell,n>0} (\ell-1) n^{r+s-1} e^{2\pi i \tau \ell n},\]
where $\delta_{a,b}$ denotes the Kronecker delta.
From this, $\tilde{g}_{r,s}^\shuffle (\tau)= \tilde{g}_{r,s}(\tau) $ holds for $r\ge2$.
See also \cite[Section 5]{B1} for more examples.
\end{example}

The definition of the shuffle regularization $\mathsf{g}^\shuffle$ of $\mathsf{g}$ is now obvious.

\begin{theorem}\label{thm:g^sh}
For each index $\bk$ whose entries are all strictly greater than $1$, we have $\tilde{g}^\shuffle_{\bk}(\tau)=\tilde{g}_\bk(\tau)$.
Furthermore, the $\Q$-linear map 
\[\mathsf{g}^\shuffle :\mathfrak{H}^1_\shuffle  \longrightarrow \Q[\pi i]\llbracket q\rrbracket  ;\quad e_{\bk}\longmapsto g_{\bk}^\shuffle\]
is an algebra homomorphism, where we set $g_{\bk}^\shuffle :=(-2\pi i)^{\wt(\bk)}\tilde{g}_{\bk}^\shuffle$ for each index $\bk$.

\end{theorem}

The shuffle regularization of $G_{\bk}(\tau)$ is then defined as follows.

\begin{theorem}\label{thm:G^sh}
For each index $\bk$, we define a $q$-series $G^\shuffle_\bk\in \mathcal{Z}[\pi i]\llbracket q\rrbracket $ by setting
\begin{equation*}\label{eq:def_G^sh} 
G^\shuffle_\bk :=  m\circ (\mathsf{Z}^\shuffle \otimes \mathsf{g}^\shuffle )\circ \Delta_G^1(e_\bk) .
\end{equation*}
Then, the $\Q$-linear map 
\[\mathsf{G}^\shuffle \colon \mathfrak{H}^1_\shuffle  \longrightarrow \mathcal{Z}[\pi i]\llbracket q\rrbracket  ;\quad e_{\bk}\longmapsto G_{\bk}^\shuffle\]
is an algebra homomorphism.
Moreover $G_{\bk}^\shuffle(\tau)$ coincides with $G_\bk(\tau)$ for any index $\bk$ whose entries are all strictly greater than $1$, where we set $G_{\bk}^\shuffle(\tau):=G_{\bk}^\shuffle\big|_{q=e^{2\pi i\tau}}$.
\end{theorem}

For details, see \cite[Theorem 4.6]{BT} and \cite[Theorem 4.10]{BT}, respectively.
Below in Theorem \ref{thm:Gamma_ME}, we provide an alternative expression of $G^\shuffle_\bk$ based on the {\em Ihara group law}, which is dual to the Goncharov coproduct.

\begin{example}
When $d=1$, the regularized Eisenstein series $G^\shuffle_k$ for $k\geq 1$ is given by
\begin{equation*}\label{eq:q-exp_single}
G^\shuffle_k= \zeta^\shuffle(k) + g_k^\shuffle.
\end{equation*}
The case $d=2$ can be found in \cite[Example 4.9]{BT}; for $r,s\ge1$ we have
\begin{equation}\label{eq:q-exp_double}
G^\shuffle_{r,s}= \zeta^\shuffle(r,s) + \sum_{p=1}^{r+s-1} C_{r,s}^p \zeta^\shuffle(p)g^\shuffle_{r+s-p} + g^\shuffle_{r,s},
\end{equation}
where $C_{r,s}^p$ is an integer defined by $C_{r,s}^p = \delta_{r,p} + (-1)^{r} \binom{p-1}{r-1} + (-1)^{p-s} \binom{p-1}{s-1}$.
Here $\delta_{a,b}$ denotes the Kronecker delta and we set $ \binom{n}{m}=0$ unless the integers $n$ and $m$ satisfy $n\geq m\geq 0$.
These explicit expressions will be effectively used in the proof of Theorem \ref{thm:main} (ii).
\end{example}

Since $\Z\tau+\Z$ is totally ordered under the order $<$, we see that, from the series expression of $G_{\bk}(\tau)$ proposed in Definition \ref{def:MES}, the multiple Eisenstein series $G_\bk(\tau)$ satisfies the stuffle relation for each index $\bk$ whose entries are all strictly greater than $1$. Note that the subspace spanned by $e_{\bk}$ with each component of $\bk$ strictly greater than $1$ forms a subalgebra of $\mathfrak{H}^1_\ast$.
From this, one can obtain a weaken version of the ``double shuffle relations'' for the $q$-series $G_{\bk}^\shuffle$, which is proposed in \cite[Theorem 1.2]{BT}.

\begin{corollary}\label{cor:dsr_mes}
For indices $\bk$ and $\bl$ whose entries are all strictly greater than $1$, we have
\[ \mathsf{G}^\shuffle (e_\bk \shuffle e_{\bl}-e_\bk\ast e_{\bl})=0.\]
\end{corollary}

Here are a few examples of linear relations obtained from Corollary \ref{cor:dsr_mes}:
\begin{align*}
0=G_4^\shuffle  - 4G_{1,3}^\shuffle, \qquad  0=G_5^\shuffle- 6G_{1,4}^\shuffle -2G_{2,3}^\shuffle .
\end{align*}
It should be noted that Corollary \ref{cor:dsr_mes} does not exhaust all linear relations of $G_\bk^{\shuffle}$'s.
For example, Kaneko's sum formulas \cite[Proposition 3 and Corollary]{Kaneko_DE} 
(see also \cite[Remark 4.16]{BT})
\begin{align} \label{eq:sum_formula}
 \sum_{j=1}^{k-2} G_{j,k-j}^{\shuffle}=G_k^\shuffle- \dfrac{(2\pi i)^2}{2(k-2)} q\frac{d}{dq} G_{k-2}^\shuffle \quad \text{for } k\ge3
\end{align}
and
\begin{equation} \label{eq:sum_formula_2}
 \sum_{\substack{j=2\\ j:{\rm even}}}^{k-2} G_{j,k-j}^{\shuffle}=\frac34 G_k^\shuffle - \dfrac{(2\pi i)^2}{2(k-2)} q\frac{d}{dq} G_{k-2}^\shuffle \quad \text{for } k\ge4 \text{ even}
\end{equation}
are not deduced from Corollary \ref{cor:dsr_mes} due to the appearance of the derivative.
Note that we have $q\frac{d}{dq}=\frac{1}{2\pi i}\frac{d}{d\tau}$ when $q=e^{2\pi i\tau}$.

Table \ref{table:dim_E_k} below lists the numerical dimensions of the $\Q$-vector spaces $\mathcal{E}_k^{\rm adm}$ spanned by $\{G_{k_1,\ldots,k_d}^\shuffle (\tau)\mid d\ge0,k_1,\ldots,k_{d-1}\ge1,k_d\ge2,k_1+\cdots+k_d=k\}$; refer to \cite[p.210]{BT}. 
Note that from $\mathfrak{H}_\shuffle^1\cong \mathfrak{H}_\shuffle^{0}[e_1]$, the whole space $\mathcal{E}:=\mathsf{G}^\shuffle(\mathfrak{H}^1)$ coincides with $\mathcal{E}^{\rm adm}[G_1^\shuffle(\tau)]$, where $\mathcal{E}^{\rm adm}=\mathsf{G}^\shuffle(\mathfrak{H}^0)$.

\begin{table}[H] 
\caption{Numerical dimensions of $\mathcal{E}^{\rm adm}_k$}
\label{table:dim_E_k}
\vspace*{1em}
\begin{tabular}{c|ccccccccccccc}
$k$ & 2 & 3 & 4 & 5 & 6 & 7 \\ \hline
$\dim_\Q \mathcal{E}_k^{\rm adm}$ & 1 & 2 & 3 & 6 & 10 & 18 
\end{tabular}
\end{table}

\begin{remark}
Corollary \ref{cor:dsr_mes} may also hold for some indices involving 1's.
In fact, Henrik Bachmann  showed that $\mathsf{G}^\shuffle (e_{k} \shuffle e_{1,\ell}-e_{k}\ast e_{1,\ell})=0$ holds for $k,\ell\ge2$ (unpublished).
\end{remark}

\subsection{Alternative expression using the Ihara group law} \label{ssc:Ihara}
We give another formulation of the $q$-series $G_\bk^{\shuffle}$, which has not  appeared in the literature.
For a $\Q$-algebra $R$, let 
\[\gp(R):=\{S\in R\llangle  X_0,X_1\rrangle^\times \mid \Delta_{\shuffle} (S)=S\otimes S\}\]
be the set of group-like elements in the algebra $R\llangle X_0,X_1\rrangle$ of noncommutative power series with coefficients in $R$, where $\Delta_\shuffle\colon R\llangle X_1,X_2\rrangle \rightarrow R\llangle X_1,X_2\rrangle \otimes_R R\llangle X_1,X_2\rrangle$ is the shuffle coproduct given by $\Delta_\shuffle(X_a)=X_a\otimes 1+1\otimes X_a$ for $a\in\{0,1\}$.
The set $\gp(R)$ forms a group with respect to the concatenation product.
It is isomorphic to the set of $R$-valued points of the affine group scheme $\gp:={\rm Spec}(\mathfrak{H}_\shuffle)$ over $\Q$.

For a power series $S\in R\llangle X_0,X_1\rrangle$ and any word $e_{a_1}\cdots e_{a_n}\in \mathfrak{H}_\shuffle$ with $a_j\in\{0,1\}$, we define $\langle S\mid e_{a_1}\cdots e_{a_n}\rangle $ to be the coefficient of the word $X_{a_1}\cdots X_{a_n}$ in the power series $S$, and then extend it linearly to the pairing  $R\llangle X_0,X_1\rrangle \times \mathfrak{H}_{\shuffle} \rightarrow R$. 
The constant term of $S$ is denoted as $\langle S\mid e_{\varnothing}\rangle$.
A power series $S\in R\llangle X_0,X_1\rrangle^\times$ is group-like if and only if $\langle S\mid w\shuffle w'\rangle =\langle S\mid w\rangle \langle S\mid w'\rangle$ holds for any $w,w'\in \mathfrak{H}_\shuffle$, or equivalently, the $\Q$-linear map $\mathfrak{H}_\shuffle \rightarrow R$ given by $w\mapsto \langle S\mid w\rangle$ is an algebra homomorphism.

For simplicity, we write $(e_{a_1}\cdots e_{a_n})^\vee = X_{a_1}\cdots X_{a_n}$ and let $\{e_0,e_1\}^\times$ denote the set of all words in $e_0$ and $e_1$, containing $e_\varnothing$. We set $e_\varnothing^\vee=1$ as convention.
Then, a power series $S\in R\llangle X_0,X_1\rrangle$ can be written as $S=\sum_{w\in \{e_0,e_1\}^\times}  \langle S\mid w\rangle w^\vee$.
Since $\mathsf{Z}^\shuffle$ is an algebra homomorphism, we have
\begin{align} \label{eq:Phi_shuffle}
\Phi_{\shuffle}=\Phi_\shuffle(X_0,X_1):=\sum_{w\in \{e_0,e_1\}^\times}  \mathsf{Z}^\shuffle (w)w^\vee 
\end{align}
as an example of a group-like element of $R\llangle X_0,X_1\rrangle$.

Let $\sigma \colon \gp(R) \rightarrow \gp(R)$ denote the shuffle antipode;
namely, the anti-automorphism given by $\sigma(X_{a_1}X_{a_2}\cdots X_{a_n})=(-1)^nX_{a_n}\cdots X_{a_2}X_{a_1}$.
For any group-like element $S\in \gp(R)$, we have $\sigma(S)S=1$.
From this reason, we write $S^{-1}$ for $\sigma(S)$ ($S\in \gp(R)$) hereafter.
In the dual side, for a word $e_{a_1}\cdots e_{a_n}\in \{e_0,e_1\}^\times$, we write 
$\epsilon(e_{a_1}\cdots e_{a_n})=(-1)^n e_{a_n}\cdots e_{a_1}$.
Then, for any $w\in \{e_0,e_1\}^\times$, we get
\begin{equation}\label{eq:inverse_cof}
\langle S^{-1}\mid w\rangle = \langle S\mid \epsilon(w)\rangle.
\end{equation}

The Ihara group law $\circ$ is defined for $A,B\in \gp(R)$ by
\begin{equation*}
A(X_0,X_1)\circ B(X_0,X_1) :=B (X_0,A X_1 A^{-1}) A \in \gp(R).
\end{equation*}
This binary operator $\circ$ equips the set $\gp(R)$ of group-like elements with a group structure.  See \cite[Proposition 2.5]{Brown} for details.
The correspondence between the Goncharov coproduct and the Ihara group law is described as follows.
From \eqref{eq:Delta_G}, we recall the algebra homomorphism $\Delta_G\colon\mathfrak{H}_{\shuffle} \to \mathfrak{H}_{\shuffle} \otimes_{\Q} \mathfrak{H}_{\shuffle}$ induced by the Goncharov coproduct on the ring of formal iterated integrals.

\begin{theorem}\label{thm:Gon_Ih}
Let $R$ be a $\Q$-algebra.
For $A,B\in \gp(R)$, define algebra homomorphisms $\alpha\colon \mathfrak{H}_\shuffle\rightarrow R$ and $\beta\colon \mathfrak{H}_\shuffle\rightarrow R$ by $\alpha(w)=\langle A \mid w \rangle$ and $\beta(w)=\langle B \mid w \rangle$ for each $w\in \mathfrak{H}_\shuffle$, respectively.
Then, we obtain
\begin{equation*}\label{eq:ihara} \langle A\circ B \mid w\rangle = m\circ (\alpha \otimes  \beta )\circ \Delta_G(w)
\end{equation*}
for every $w\in \mathfrak{H}_\shuffle$.
\end{theorem}

A sketch of the proof of the above identity can be found in \cite[Section 2]{Brown}, although the setup is different.
We outline the proof in our setting in Appendix \ref{appendix:Ihara-Goncharov} for convenience.

\begin{corollary}\label{cor:Gon_Ih}
Let $A,B,\alpha,\beta$ be as in Theorem~$\ref{thm:Gon_Ih}$.
Suppose that $\alpha(e_0)=\beta(e_0)=0$.
Then, for each index $\bk$ of depth $d$, we have
\[ \langle A\circ B \mid e_\bk \rangle =  \beta(e_{\bk}) + \sum_{j=1}^{d-1} \sum_{\substack{\bl\in \N^{j}\\ 0< \wt(\bl)<k}} c_{\bk,\bl} \,\alpha(e_\bl) + \alpha(e_\bk)\]
for some $c_{\bk,\bl}\in \mathcal{B}_{k-\wt(\bl)}$, where $ \mathcal{B}_k$ denotes the $\Q$-vector space spanned by all $\beta(e_\bk)$ with $\wt(\bk)=k$.
\end{corollary}
\begin{proof}
 Since $\mathfrak{H}_\shuffle\cong \mathfrak{H}_\shuffle^1[e_0]$ holds, we can verify $\gamma\circ {\rm reg}_0=\gamma$ for $\gamma\in\{\alpha,\beta\}$ from the assumption, where ${\rm reg}_0\colon\mathfrak{H}_\shuffle\rightarrow \mathfrak{H}_\shuffle^1$ is defined in \eqref{eq:reg_0}.
Therefore, we have 
\begin{align*}
    \langle A\circ B \mid e_\bk \rangle=m\circ (\alpha\otimes \beta)\circ \Delta_G(e_\bk)=m\circ (\alpha\otimes \beta)\circ \Delta_G^1(e_\bk).
\end{align*}
The desired result is then obtained from \eqref{eq:Delta_G^1}.
\end{proof}

Now let us consider
\begin{equation}\label{eq:Gamma_MEMD}
\Gamma_\mathrm{MD} (X_0,X_1):=\sum_{w\in \{e_0,e_1\}^\times}  \mathsf{g}^\shuffle (w)w^\vee \quad \text{and}\quad 
 \Gamma_\mathrm{ME}(X_0,X_1) :=\sum_{w\in \{e_0,e_1\}^\times} \mathsf{G}^\shuffle (w)w^\vee,
\end{equation}
where the algebra homomorphisms $\mathsf{g}^\shuffle \colon\mathfrak{H}_\shuffle^1  \rightarrow \Q[\pi i]\llbracket q\rrbracket $ and $\mathsf{G}^\shuffle\colon\mathfrak{H}_\shuffle^1  \rightarrow \mathcal{Z}[\pi i]\llbracket q\rrbracket $, defined in Theorems \ref{thm:g^sh} and  \ref{thm:G^sh}, are extended to
\[\mathsf{g}^\shuffle \colon\mathfrak{H}_\shuffle  \longrightarrow \Q[\pi i]\llbracket q\rrbracket \quad \mbox{and} \quad \mathsf{G}^\shuffle \colon\mathfrak{H}_\shuffle  \longrightarrow \mathcal{Z}[\pi i]\llbracket q\rrbracket  \]
by setting $\mathsf{g}^\shuffle(e_0)=0$ and $\mathsf{G}^\shuffle(e_0)=0$, respectively. 
Here the subscript $\mathrm{MD}$ stands for the ``\underline{m}ultiple \underline{d}ivisor sum,'' whereas $\mathrm{ME}$ stands for the ``\underline{m}ultiple \underline{E}isenstein series.''

\begin{theorem}\label{thm:Gamma_ME} 
The power series $\Gamma_\mathrm{MD}$ and $\Gamma_\mathrm{ME}$ are elements in $\gp(\mathcal{Z}[\pi i]\llbracket q\rrbracket )$ and satisfy
\[
\Gamma_\mathrm{ME}(X_0,X_1)  = \Gamma_\mathrm{MD}(X_0,X_1)\circ \Phi_\shuffle (X_0,X_1),
\]
where $\Phi_\shuffle$ is defined as \eqref{eq:Phi_shuffle}.
\end{theorem}

The formula in Theorem \ref{thm:Gamma_ME} was independently pointed out by Hidekazu Furusho and Francis Brown.

\begin{proof}
The first statement is a consequence of Theorems \ref{thm:g^sh} and \ref{thm:G^sh}.
For the latter statement, let $\bk$ be an index.
By definition, we have $G_{\bk}^{\shuffle}=\langle \Gamma_{\mathrm{ME}}\mid e_{\bk} \rangle$.
On the other hand, Corollary \ref{cor:Gon_Ih} shows that $ m\circ (\mathsf{g}^\shuffle \otimes  \mathsf{Z}^\shuffle )\circ \Delta_G^1(e_\bk )=\langle \Gamma_{\mathrm{MD}}\circ \Phi_\shuffle \mid e_{\bk}\rangle$ holds, where the left-hand side is, by definition, equal to $G_{\bk}^{\shuffle}$.
Hence, for any index $\bk$, we get 
\begin{align*} \label{eq:G_shuffle_Ihara}
   \langle \Gamma_{\mathrm{ME}}\mid e_{\bk} \rangle =\langle \Gamma_{\mathrm{MD}}\circ \Phi_\shuffle \mid e_{\bk} \rangle. 
\end{align*}
The desired identity is then deduced from the following fact.
Consider $A,B\in\Pi(R)$ satisfying $\alpha(e_0)=\beta(e_0)=0$, where $\alpha$ and $\beta$ are as in Theorem \ref{thm:Gon_Ih}.
Then $A$ equals $B$ if and only if $\langle A\mid e_{\bk} \rangle$ equals $\langle B \mid e_{\bk} \rangle$ for all indices $\bk$, since $\mathfrak{H}_\shuffle\cong \mathfrak{H}_\shuffle^1[e_0]$ holds.
\end{proof}

We remark that every element of $\gp(R)$ has its inverse in $\gp(R)$ with respect to the Ihara group law.
As an example, the inverse of $\Phi_\shuffle$ (defined as \eqref{eq:Phi_shuffle}) with respect to the Ihara group law, denoted by $\Phi^{\rm inv}_{\shuffle} \in \gp (\mathcal{Z})$, can be computed inductively by the equation $\Phi_\shuffle\circ \Phi_\shuffle^{{\rm inv}}=1$, or alternatively, by the equation $m\circ (\mathsf{Z}^\shuffle \otimes  \mathsf{Z}^{\rm inv} )\circ \Delta_G^1(e_\bk) =0$ for non-empty index $\bk$, where the $\Q$-linear map $\mathsf{Z}^{\rm inv}:\mathfrak{H}\rightarrow \mathcal{Z}$ is defined for $w\in \mathfrak{H}$ by
\[\mathsf{Z}^{\rm inv} (w)=\langle \Phi_{\shuffle}^{\rm inv}\mid w\rangle.\]
Note that the inverse $\Phi_\shuffle^{\rm inv}$ with respect to $\circ$ is different from the inverse  $\Phi_\shuffle^{-1}$ as an element of $R\llangle X_0,X_1\rrangle^\times$.

As a direct consequence of Theorem \ref{thm:Gamma_ME}, we obtain the following expressions of $G^\shuffle_\bk(\tau)$ and $g^\shuffle_\bk(\tau)$:
for each index $\bk$ of depth $d$, we have
\begin{align}
\label{eq:G_g} G_{\bk}^\shuffle(\tau) &= \zeta^\shuffle(\bk)+\sum_{j=1}^{d-1} \sum_{\substack{\bl\in \N^{j}\\ 0< \wt(\bl)<k}} a_{\bk,\bl}\, g_{\bl}^\shuffle(\tau)+g_\bk^\shuffle(\tau),\\
\label{eq:g_G}  g_{\bk}^\shuffle(\tau) &= \mathsf{Z}^{\rm inv} (e_\bk)+\sum_{j=1}^{d-1} \sum_{\substack{\bl\in \N^{j}\\ 0< \wt(\bl)<k}} b_{\bk,\bl}\, G_{\bl}^\shuffle(\tau)+G_\bk^\shuffle(\tau)
\end{align}
for some $a_{\bk,\bl} \in \mathcal{Z}_{k-\wt(\bl)}$ and $b_{\bk,\bl} \in \mathcal{Z}_{k-\wt(\bl)}$.

\subsection{Comparison between linear relations of $G^\shuffle_{\bk}$ and $\tilde{g}^\shuffle_{\bk}$}
\label{ssc: comparison_LR}

Based on the expressions \eqref{eq:G_g} and \eqref{eq:g_G}, we now compare the linear relations among $G^\shuffle_{\bk}$ with those among $\tilde{g}^\shuffle_{\bk}=g^\shuffle_{\bk}/(2\pi i)^{\wt(\bk)}$. 
Although the result in this subsection will not be used in the remainder of the paper, it may provide an alternative approach to investigating the algebraic structure of multiple Eisenstein series.

We begin by the following lemma.

\begin{lemma}\label{lem:rational_structure}
For a non-negative integer $k$ and a $\Q$-algebra $R$ contained in $\C$, let $\mathcal{D}_{\le k}^R$ be the $R$-module generated by all $\tilde{g}^\shuffle_{\bk}$ of weight less than or equal to $k$.
	Then, the $\Q$-linear map
	\[
	\mathcal{D}^\Q_{\le k}\otimes_{\Q}R\longrightarrow \mathcal{D}^R_{\le k};\quad g\otimes r\longmapsto rg
	\]
	is an isomorphism of $R$-modules.
	\label{lemforHI2}
	\end{lemma}
	\begin{proof}
		It suffices to show the injectivity of the map under consideration, because the surjectivity is clear.
		Since $\tilde{g}_{\bk}^\shuffle \in \Q\llbracket q\rrbracket $, we have the following commutative diagram:
		\[
		\xymatrix{
		\mathcal{D}^\Q_{\le k}\otimes_{\Q}R\ar[rrr]\ar@{^{(}->}[d]& & &\mathcal{D}^R_{\le k}\ar@{^{(}->}[d]\\
		\Q\llbracket q\rrbracket \otimes_{\Q}\C\ar[rrr]& & &\C\llbracket q\rrbracket 
.		}
		\]
		Then, since the lower horizontal homomorphism is injective, so is the upper horizontal homomorphism.
	\end{proof}

\begin{proposition}\label{halfimplicationofBKconj}
Let $k$ be a positive integer. Then a linear relation 
$\sum_{\wt(\bk)=k} c_\bk G^\shuffle_\bk=0$ $(c_{\bk}\in \Q)$ of $G^\shuffle_{\bk}$'s implies the corresponding congruence $\sum_{\wt(\bk)=k} c_{\bk} \tilde{g}^\shuffle_\bk \equiv 0 \mod  \mathcal{D}_{\le k-1}^\Q$.  The converse implication holds if the sum $\sum_{k\geq 0}\mathcal E_{k}^{\C}$ of the $\C$-vector spaces $\mathcal E_{k}^{\C}$ spanned by all $G_{\bk}^\shuffle$'s with $\wt(\bk)=k$ is a direct sum.
\end{proposition}

\begin{proof}
Suppose that the identity $\sum_{\wt(\bk)=k} c_\bk G^\shuffle_\bk =0 $ holds. 
By \eqref{eq:G_g}, we obtain the congruence
\begin{equation}\label{congruenceofg}
\sum_{\wt(\bk)=k} c_{\bk} \tilde{g}^\shuffle_\bk\equiv 0 \mod  \mathcal{D}_{\le k-1}^\C.
\end{equation}
In other words, the left-hand side of the congruence above is zero in the $\C$-vector space $\mathcal{D}_{\le k}^\C /\mathcal{D}_{\le k-1}^\C $. 
On the other hand, we have a natural inclusion and an isomorphism of vector spaces
\[\mathcal{D}_{\le k}^\Q /\mathcal{D}_{\le k-1}^\Q \hookrightarrow \left(\mathcal{D}_{\le k}^\Q /\mathcal{D}_{\le k-1}^\Q\right)\otimes_{\Q}\C\xrightarrow{\sim} \mathcal{D}_{\le k}^\C /\mathcal{D}_{\le k-1}^\C\]
by Lemma \ref{lemforHI2}.
Hence, the left-hand side of \eqref{congruenceofg} is already zero in $\mathcal{D}_{\le k}^\Q /\mathcal{D}_{\le k-1}^\Q$.
This gives the first assertion.
		
Next, we show the converse implication assuming the linearly independence over $\C$ of multiple Eisenstein series of different weights.
Suppose that there exists $\xi \in \mathcal{D}_{\le k-1}^\Q$ such that $\sum_{\wt(\bk)=k} c_{\bk} \tilde{g}^\shuffle_\bk =\xi$. 
Multiplying both sides by $(-2\pi i)^k$ and then substituting \eqref{eq:g_G} into it, we see that 
\begin{equation*}\label{eq:G_cong}
\sum_{\wt(\bk)=k} c_{\bk} G^\shuffle_\bk \equiv 0 \mod  \sum_{\ell=0}^{k-1} \mathcal{E}_{\ell}^{\C}.
\end{equation*}
From the assumption, we observe that the intersection of $\mathcal{E}_{k}^{\C}$ and $\sum_{\ell=0}^{k-1} \mathcal{E}_{\ell}^\C$ equals $\{0\}$.
Hence, the congruence above implies the identity $\sum_{\wt(\bk)=k} c_{\bk} G^\shuffle_\bk= 0$.
This completes the proof of the proposition.
\end{proof}

Let $\mathcal{D}_{\le k}^{\rm adm}$ denote the $\Q$-vector space spanned by 
\begin{align*}
\{\tilde{g}^\shuffle_{k_1,\ldots,k_d} (\tau)\mid d\ge0,k_1,\ldots,k_{d-1}\ge1,k_d\ge2,k_1+\cdots+k_d\le k\}.
\end{align*}
Proposition \ref{halfimplicationofBKconj} implies that the map
\[\mathcal{E}_k^{\rm adm}\longrightarrow \mathcal{D}_{\le k}^{\rm adm} / \mathcal{D}_{\le k-1}^{\rm adm}; \quad G_{\bk}^\shuffle \longmapsto \tilde{g}_{\bk}^\shuffle \mod \mathcal{D}_{\le k-1}^{\rm adm}\]
is an isomorphism, if the sum $\sum_{k \ge 0} \mathcal{E}_k^{\C}$ is a direct sum. 
As supporting evidence, it was numerically observed in \cite[Section 5]{B1} that the equality $\dim_\Q \mathcal{D}_{\le k}^{\rm adm} / \mathcal{D}_{\le k-1}^{\rm adm} = \dim_\Q \mathcal{E}_k^{\rm adm}$ holds up to weight $8$; see also Table \ref{table:dim_E_k}.

\begin{remark}
Based on Bachmann's  stuffle regularization of multiple Eisenstein series developed in \cite[Theorem 2]{B1}, 
Bachmann and Burmester  introduced in \cite{BB} another variant of multiple Eisenstein series as $q$-series in $\Q\llbracket q\rrbracket $, which they call the \emph{combinatorial $($bi-$)$multiple Eisenstein series}.
They conjectured in \cite[Remark~6.11]{BB} that their combinatorial (bi-)multiple Eisenstein series of different weights are linearly independent over $\Q$. 
This conjecture is closely related to our assumption that $\sum_{k \ge 0} \mathcal{E}_k^{\mathbb{C}}$ is a direct sum. 

More precisely, fix an element $\phi \in \mathrm{DMR}_0(\mathbb{Q})$ in Racinet's double shuffle group \cite[D\'efinition 3.2, Th\'eor\`eme I]{Racinet02}, which is a subgroup of $\gp(\mathbb{Q})$ with respect to the Ihara group law, characterized by the extended double shuffle relation. Such an element $\phi$ exists but is not unique; see \cite[Remark~4.1]{BB}. 
Then, for each index $\boldsymbol{k}$, one can define the shuffle version of the combinatorial multiple Eisenstein series $G_\bk^{\phi}$ by
\[
  G_{\boldsymbol{k}}^{\phi}
  := \langle \Gamma_{\mathrm{MD}} \circ \phi \mid e_{\bk}\rangle
  \in \Q\llbracket q\rrbracket .
\]
Note that the coefficients of $\phi$ satisfy the same relations over $\Q$ as multiple zeta values, except for $\langle  \phi\mid e_2 \rangle = 0$, which corresponds to setting ``modulo $\zeta(2) $.'' 
It is thus reasonable to expect that $G_{\bk}^{\phi}$ satisfies the same relations as $G_{\bk}^{\shuffle}$, if one takes Proposition \ref{halfimplicationofBKconj} into accounts. 
If this is the case, the equality $\sum_{k \ge 0} \mathcal{E}_k^{\mathbb{C}}=\bigoplus_{k \ge 0} \mathcal{E}_k^{\mathbb{C}}$ would follow provided that the (shuffle version of) combinatorial multiple Eisenstein series $G_\bk^{\phi}$ of different weights are linearly independent over $\Q$.
\end{remark}

\section{Symmetric multiple Eisenstein series}
We now introduce our main object, the {\em symmetric multiple Eisenstein series}.
Since the definition is parallel to that of symmetric multiple zeta values, we first review their definition and basic properties in Section~\ref{ssc:review_SMZV}, and then establish analogous results for symmetric multiple Eisenstein series in Section~\ref{ssc:SMES}.
We also give an observation on the dimension of the $\Q$-vector space spanned by symmetric multiple Eisenstein series in Section~\ref{ssc:num_obs}.

\subsection{Review on symmetric multiple zeta values} \label{ssc:review_SMZV}
As a starting point, we recall some basic facts on symmetric multiple zeta values, including their series representations in terms of 
Kontsevich's order and their fundamental relations, namely, the stuffle relation and the linear shuffle relation.
We also briefly mention several topics related to them, highlighting their significance.

We define an order $\prec$ on the set of nonzero integers by
\begin{equation}\label{eq:order_int}
1 \prec 2 \prec 3 \prec \cdots \prec (\infty = -\infty) \prec \cdots \prec -3 \prec -2 \prec -1 ,
\end{equation}
where the largest element $\infty$ is identified with the smallest element $-\infty$.
In this paper, the order $\prec$ is called \emph{Kontsevich's order}, as suggested by Maxim Kontsevich to Don Zagier in a private communication.
With this, for positive integers $k_1,\ldots,k_d \ge 2$, we define 
$\zeta^{S}(k_1,\ldots,k_d)$ by the absolutely convergent series
\begin{equation*}\label{eq:series_SMZV}
\zeta^{S}(k_1,\ldots,k_d)
=
\sum_{\substack{n_1 \prec \cdots \prec n_d \\ n_1,\ldots,n_d \in \mathbb{Z}\setminus\{0\}}}
n_1^{-k_1} \cdots n_d^{-k_d}.
\end{equation*}
Then, dividing the sum into partial sums corresponding to the index $j$ where the jump $n_j\prec (\infty=-\infty) \prec n_{j+1}$ occurs, we obtain
\begin{equation*}\label{eq:sym_mzv} 
\zeta^{S}(k_1,\ldots,k_d)=\sum_{j=0}^d (-1)^{k_{j+1}+\cdots+k_d} \zeta(k_1,\ldots,k_j)\zeta(k_d,\ldots,k_{j+1}).
\end{equation*}
From the definition \eqref{eq:sym_mzv_reg}, we have $\zeta^S(k_1,\ldots,k_d)=\zeta^{\shuffle,S}(k_1,\ldots,k_d)$ for $k_1,\ldots,k_d\geq 2$.
Thus, we obtain the stuffle relation for the symmetric multiple zeta value
$\zeta^{\shuffle,S}(\bk)$ whenever all components of $\bk$ are strictly greater than~$1$.
This is because Kontsevich's order is a total order on the set $\mathbb{Z}\setminus\{0\}$; hence, the standard argument described in Section~\ref{subsec:alg_setup} applies.

For the \emph{linear shuffle relation}, we use the following expression of $\zeta^{\shuffle,S}(\bk)$ (see also \cite{Jarossay}).

\begin{proposition}\label{prop:sym_MZV_gen}
For positive integers $k_1,\ldots,k_d$, we have
\begin{align*}
 \zeta^{\shuffle,S} (k_1,\ldots,k_d)&=
\langle \Phi_\shuffle X_1 \Phi_\shuffle^{-1}\mid e_1e_0^{k_1-1}\cdots e_1e_0^{k_d-1}e_1\rangle .
\end{align*}
\end{proposition}

\begin{proof}
For $S_1,S_2\in \gp$ and a word $w\in \{e_0,e_1\}^\times$, we have $\langle S_1S_2\mid w\rangle = \displaystyle\sum_{uv=w}\langle S_1\mid u\rangle \langle S_2\mid v\rangle$, where the sum runs over all deconcatenations of $w$, including $w=e_\varnothing w= we_\varnothing $.
From \eqref{eq:inverse_cof}, we obtain
\begin{align*}
\langle \Phi_\shuffle X_1 \Phi_\shuffle^{-1}\mid e_\bk e_1\rangle& = \sum_{j=0}^d \langle \Phi_\shuffle \mid e_1e_0^{k_1-1}\cdots e_1e_0^{k_j-1}\rangle  \langle X_1\mid e_1\rangle \langle \Phi_\shuffle^{-1} \mid e_0^{k_{j+1}-1}\cdots e_1e_0^{k_d-1}e_1\rangle\\
&= \sum_{j=0}^d \zeta^{\shuffle}(k_1,\ldots,k_j) (-1)^{k_{j+1}+\cdots+k_d} \zeta^{\shuffle }(k_d,\ldots,k_{j+1}).
\end{align*}
Hence the desired result follows.
\end{proof}

From this expression, the linear shuffle relation for the symmetric multiple zeta value is a consequence of the following property of group-like elements.

\begin{lemma}\label{lem:linear_shuffle_general}
Let $R$ be a $\Q$-algebra.
For $A\in \gp(R)$, define an algebra homomorphism $\alpha\colon \mathfrak{H}_\shuffle \to R; \ w\mapsto \langle A\mid w\rangle $.
For a word $w\in \{e_0,e_1\}^\times$, we define the `symmetrization' $\alpha^S(w)$ of $\alpha(w)$ by 
\[\alpha^S(w)=  \langle A X_1 A^{-1}\mid we_1\rangle.\]
Then, for any positive integers $k_1,\ldots,k_d$ and $j=0,1,\ldots,d-1$, we have 
\[\alpha^{S} (e_{k_1,\ldots,k_j} \shuffle e_{k_{j+1},\ldots,k_d} ) = (-1)^{k_{j+1}+\cdots+k_d} \alpha^{S}(e_{k_1,\ldots,k_j,k_{d},\ldots,k_{j+1}}).\]
\end{lemma}
\begin{proof}
This can be shown by using \cite[Proposition~9.7]{Kaneko}.
Here, we repeat the proof following the idea used in \cite{Hirose,Jarossay,SingerZhao,Tasaka}.

Since $\Delta_\shuffle(X_1)=X_1\otimes 1+1\otimes X_1$, one has 
$\Delta_\shuffle (AX_1A^{-1})=AX_1A^{-1}\otimes 1 + 1\otimes AX_1A^{-1}$.
This means that the power series $AX_1A^{-1}\in R\llangle X_0,X_1\rrangle$ is Lie-like.
In other words, for non-empty words $w_1,w_2\in \{e_0,e_1\}^\times$, we have $\langle AX_1A^{-1}\mid w_1\shuffle w_2\rangle =0$. 
For words $u\in \mathfrak{H}^1$ and $v\in \mathfrak{H}$, one can compute
\begin{align*}
\alpha^S( u\shuffle e_1 v) &= \left\langle AX_1A^{-1} \mid (u\shuffle e_1v)e_1 \right\rangle \\
&=   \left\langle AX_1A^{-1} \mid u e_1 \epsilon(e_1v)  \right\rangle =-\alpha^S( u e_1\epsilon(v)),
\end{align*}
where $\epsilon$ is introduced in \eqref{eq:inverse_cof}. The second equality follows from the identity in \cite[Lemma 19]{Hirose}.
By setting $u=e_{k_1,\ldots,k_j}\in \mathfrak{H}^1$ and $e_1v=e_{k_{j+1},\ldots,k_d}\in \mathfrak{H}$, we obtain the desired result.
\end{proof}

Applying Lemma \ref{lem:linear_shuffle_general} to $\Phi_\shuffle$, we obtain the linear shuffle relation for the symmetric multiple zeta value.
Let
\[\mathsf{Z}^{\shuffle,S}\colon \mathfrak{H}^1\longrightarrow \mathcal{Z};\quad e_\bk \longmapsto \zeta^{\shuffle,S}(\bk)=
\langle \Phi_\shuffle X_1 \Phi_\shuffle^{-1}\mid e_{\bk}e_1\rangle.\]
Note that this is {\em not} an algebra homomorphism. 
Then, for positive integers $d,k_1,\ldots,k_d$ and $j=0,1,\ldots,d-1$, we have
\[ \mathsf{Z}^{\shuffle,S} \big(e_{k_1,\ldots,k_j} \shuffle e_{k_{j+1},\ldots,k_d}\big)= (-1)^{k_{j+1}+\cdots+k_d}\mathsf{Z}^{\shuffle,S}\big(e_{k_1,\ldots,k_j,k_{d},\ldots,k_{j+1}}\big). \]

We remark that, for each index $\bk$, the symmetric multiple zeta values can be written as a $\mathbb{Q}$-linear combination of multiple zeta values due to the shuffle relation.
Conversely, Yasuda showed in \cite[Theorem 6.1]{Yasuda} that every multiple zeta value can be expressed as a $\mathbb{Q}$-linear combination of symmetric multiple zeta values. In contrast, we will observe in Section~\ref{ssc:num_obs} that not every multiple Eisenstein series can be expressed in terms of symmetric multiple Eisenstein series.

\begin{remark}\label{rem:KZ}
The symmetric multiple zeta values were first introduced by Kaneko and Zagier in \cite{KZ} as a real counterpart of finite multiple zeta values; see also \cite[Section 9]{Kaneko} and \cite{Z16}.
For an index $\bk = (k_1,\ldots,k_d)$, the {\em finite multiple zeta value} is defined by
\begin{align*}
\zeta^{\mathcal{A}}(\bk)
=
\Bigl(
\sum_{0 < n_1 < \cdots < n_d < p}
n_1^{-k_1} \cdots n_d^{-k_d}
\ \bmod p
\Bigr)_p
\end{align*}
as an element of the $\mathbb{Q}$-algebra
$
\mathcal{A}
=
\prod_p \mathbb{F}_p \Big/ \bigoplus_p \mathbb{F}_p$ sometimes called the `poor man's ad\`ele ring,'
where $p$ runs over all prime numbers.
Kaneko and Zagier conjectured that, for rational numbers $\{a_{\bk}\}_{\bk} \subset \mathbb{Q}$, the linear relation
$\sum_{\bk} a_{\bk} \, \zeta^{\mathcal{A}}(\bk) = 0$ in $\mathcal{A}$
holds if and only if the congruence $\sum_{\bk} a_{\bk} \, \zeta^{\shuffle,S}(\bk)
\equiv 0 \pmod{\zeta(2)\mathcal{Z}}$
holds.
Generalizations of Kaneko and Zagier's conjecture have been actively studied; see
\cite{OnoSekiYamamoto21,TakeyamaTasaka,Tasaka} and the references therein.

Their connection with modular forms was also observed by Kaneko and Zagier.
More precisely, for even integers $k \ge 4$, they proposed a conjectural dimension formula
\[
\dim_{\mathbb{Q}}
\langle
\zeta^{\mathcal{A}}(\bk)
\mid
\bk \in \mathbb{N}^3,\ \mathrm{wt}(\bk) = k
\rangle_{\mathbb{Q}}
\stackrel{?}{=}
\frac{k}{2} - 2
- \dim_{\mathbb{C}} S_k(\mathrm{SL}_2(\mathbb{Z})),
\]
where $S_k(\mathrm{SL}_2(\Z))$ is the space of cusp forms of weight $k$ for $\mathrm{SL}_2(\Z)$.
For this, one can observe that the $\mathbb{Q}$-vector space on the left-hand side above is generated by the set $B_k=\{\zeta^{\mathcal{A}}(1,r,s)\mid r+s=k-1, r,s\ge2,s\equiv 0\bmod{2}\}$. 
Since $|B_k| = k/2 - 2$, the above dimension conjecture predicts that the elements of
$B_k$ satisfy $\dim_{\mathbb{Q}} S_k^{\mathbb{Q}}(\mathrm{SL}_2(\mathbb{Z}))$
independent linear relations over $\mathbb{Q}$.
This suggests existence of linear relations among finite (or, via the conjectural correspondence, symmetric) \emph{triple} zeta values arising from cusp forms.
The first example is as follows:
\[1470\zeta^{\mathcal{A}}(1,9,2) -1059\zeta^{\mathcal{A}}(1,7,4)+ 669\zeta^{\mathcal{A}}(1,5,6) -238\zeta^{\mathcal{A}}(1,3,8)=0.\]
\end{remark}

\begin{remark}
There are three more remarks on the symmetric form, the right-hand side of \eqref{eq:sym_mzv_reg}.
Firstly, it appears in the explicit formula for the limit of the multiple harmonic $q$-sum $H_n(k_1,\ldots,k_d;q)=\sum_{0<n_1<\cdots<n_d<n}[n_1]_q^{-k_1}\cdots [n_d]_q^{-k_d}$, where $[n]_q:=1+q+\cdots+q^{n-1}=\frac{q^n-1}{q-1}$ is the $q$-integer.
Indeed, for each index $\bk=(k_1,\ldots,k_d)$, the limit $\xi(\bk):=\lim_{n\rightarrow \infty} H_n(\bk;e^{2\pi i/n})$ exists and satisfies
\[ \xi(k_1,\ldots,k_d) \equiv \sum_{j=0}^d (-1)^{k_{j+1}+\cdots+k_d} \zeta^\ast(k_1,\ldots,k_j)\zeta^\ast(k_d,\ldots,k_{j+1})  \mod \pi i,\]
where $\zeta^\ast(\bk)$ denotes the stuffle regularization of $\zeta(\bk)$; see \cite[Theorem 2.10]{BTT} and \cite[the $N=1$ case of Theorem 2.2]{Tasaka}.
Note that the value $\xi(\bk)$ coincides with Hirose's refined symmetric multiple zeta value $\zeta^{RS}(\bk)$ defined by iterated integrals along a loop based at a tangential base point, which gives an iterated integral representation of the symmetric multiple zeta value. See \cite[Remark 11]{Hirose} for details.

Secondly, if multiple zeta values in \eqref{eq:sym_mzv_reg} are replaced with Deligne's $p$-adic multiple zeta values $\zeta_p(\bk)$, then we get a connection with multiple harmonic sums as observed by Akagi, Hirose and Yasuda \cite{AHY} and proved by Jarossay \cite[Corollary 2]{Jarossay1}:
\[ \sum_{j=0}^d (-1)^{k_{j+1}+\cdots+k_d} \zeta_p(k_1,\ldots,k_j)\zeta_p(k_d,\ldots,k_{j+1}) \equiv \sum_{0<n_1<\cdots<n_d<p} n_1^{-k_1}\cdots n_d^{-k_d} \mod p.\]
This congruence relation provides a deep understanding and plays a crucial role in the study of finite multiple zeta values and its generalization; see \cite{OnoSekiYamamoto21,Rosen,TakeyamaTasaka} for further references.

Finally, Komori showed in \cite[Theorem 1.3]{Komori} that a functional version of the symmetric sum \eqref{eq:sym_mzv_reg}, or rather, the function 
\[\sum_{j=0}^d (-1)^{s_{j+1}+\cdots+s_d} \zeta(s_1,\ldots,s_j)\zeta (s_d,\ldots,s_{j+1}) \] 
of several complex variables is entire, where $(-1)^s$ denotes $e^{\pi i s}$.
This is surprising because the multiple zeta function $\zeta(s_1,\ldots,s_d)$ has singularities on infinitely many hyperplanes, as observed in \cite[Theorem 1]{AET} for example, and all these singularities are resolved by the sum of products appearing on the right-hand side of \eqref{eq:sym_mzv_reg}.
See also \cite{KOOT,OnoYamamoto} for related works.
\end{remark}

\subsection{Symmetric multiple Eisenstein series and its basic properties} \label{ssc:SMES}
Along a procedure analogous to the symmetric multiple zeta values, we give a series representation and prove the linear shuffle relation for the symmetric multiple Eisenstein series.

\begin{definition}\label{def:smes}
For each index $\bk=(k_1,\ldots,k_d)$, the symmetric multiple Eisenstein series $G_\bk^{\shuffle,S}$ is defined by
\[G^{\shuffle,S}_{\bk}= \sum_{j=0}^d (-1)^{k_{j+1}+\cdots+k_d}G^{\shuffle}_{k_1,\ldots,k_j} G^{\shuffle}_{k_d,\ldots,k_{j+1}}\in \mathcal{Z}[\pi i]\llbracket q\rrbracket .\]
\end{definition}

Similarly to $\mathsf{Z}^{\shuffle,S}$, let us define a $\Q$-linear map $\mathsf{G}^{\shuffle,S}\colon \mathfrak{H}^1\rightarrow \mathcal{Z}[\pi i]\llbracket q\rrbracket $ by setting $\mathsf{G}^{\shuffle,S}( e_\bk )= G^{\shuffle,S}_\bk$.
As an example, for $d=1$ and $k\in \mathbb{N}$, we have
\begin{equation}\label{eq:dep1} 
\mathsf{G}^\shuffle(e_k)=G^{\shuffle,S}_k  = \begin{cases} 2G_k^\shuffle & k:{\rm even},\\  0 & k:{\rm odd}.\end{cases}
\end{equation}

For the series expression of $G^{\shuffle,S}_{\bk}$, we define a total order $ \prec $ on $ \bigl(\Z\tau + \Z \bigr)\setminus \{0\} $ as a generalization of Kontsevich's order on $\Z\setminus\{0\}$. 
Recall the notation $P_0$ and $P_1$ from \eqref{eq:P}.  
For $ \lambda \in \bigl(\Z\tau + \Z \bigr)\setminus \{0\} $, we write $ \lambda < 0 $ if $ \lambda \notin P := P_0 \cup P_1 $.  

\begin{definition}
For $ \lambda, \mu \in \bigl(\Z\tau + \Z \bigr)\setminus \{0\} $, we write $\lambda \prec \mu$ if either $0 < \lambda < \mu$,  $\mu < 0 < \lambda$ or $\lambda < \mu <0$ holds.
\end{definition}

This defines a total order $\prec$ on $\bigl(\Z\tau + \Z \bigr)\setminus \{0\}$, which is obtained from the same idea as in Kontsevich's order \eqref{eq:order_int};
more precisely, we identify the maximal element $\infty\tau + \infty$ and the minimal element $-\infty\tau - \infty$ in the totally ordered set $(\mathbb{Z}\tau + \mathbb{Z}, <)$.
With this interpretation, Kontsevich's order \eqref{eq:order_int} appears as a special case of the above order.

Let us illustrate an example.
We first decompose the set
\[\{(\lambda,\mu) \in \bigl((\Z\tau+\Z)\setminus\{0\}\bigr)^2\mid \lambda\prec \mu\}\] 
into the disjoint union of the following three subsets
\begin{equation*}\label{eq:decomp_prec}
\begin{aligned}
&\{\, (\lambda, \mu) \in \bigl((\Z\tau + \Z)\setminus \{0\}\bigr)^2 \mid 0 < \lambda < \mu \,\},\\
&\{\, (\lambda, \mu) \in \bigl((\Z\tau + \Z)\setminus \{0\}\bigr)^2 \mid \mu < 0 < \lambda \,\},\\
&\{\, (\lambda, \mu) \in \bigl((\Z\tau + \Z)\setminus \{0\}\bigr)^2 \mid \lambda < \mu < 0 \,\}.
\end{aligned}
\end{equation*}
From this, for $r,s\ge3$, we obtain
\[ \sum_{\substack{\lambda\prec \mu\\ \lambda,\mu \in (\Z\tau+\Z) \setminus\{0\}}} \frac{1}{\lambda^{r} \mu^{s}}  = G_{r,s}(\tau)+G_{r}(\tau)\cdot (-1)^{s}G_{s}(\tau)+(-1)^{r+s}G_{s,r}(\tau),\]
which coincides with the symmetric double Eisenstein series $G^{\shuffle,S}_{r,s}(\tau)$. The condition $r,s\geq 3$ is imposed for absolute convergence.
In general, we obtain the following result. Recall that $\mathbb{Z}_M$ denotes $\{m\in \mathbb{Z}\mid -M\leq m\leq M\}$ for $M\geq 0$.

\begin{proposition}\label{prop:stuffle}
For positive integers $k_1,\ldots,k_d\ge2$, we have
\begin{equation}\label{eq:G^S}
\begin{aligned}
G^{\shuffle,S}_{k_1,\ldots,k_d}(\tau)& = \lim_{M\rightarrow\infty}\lim_{N\rightarrow \infty} \sum_{\substack{\lambda_1\prec \cdots \prec \lambda_d\\ \lambda_1,\ldots,\lambda_d \in (\Z_M\tau+\Z_N) \setminus\{0\}}} \frac{1}{\lambda_1^{k_1}\cdots \lambda_d^{k_d}} .
\end{aligned}
\end{equation}
Hence, for indices $\bk$ and $\bl$ whose entries are all strictly greater than $1$, we have
\[  \mathsf{G}^{\shuffle,S} (e_\bk\ast e_{\bl})= \mathsf{G}^{\shuffle,S} (e_\bk) \mathsf{G}^{\shuffle,S} (e_{\bl}).\]
\end{proposition}
\begin{proof}
The fact that symmetric multiple Eisenstein series satisfy the stuffle relation follows directly from the series representation \eqref{eq:G^S} together with the general framework of the stuffle relation described in Section~\ref{subsec:alg_setup}.
So, it suffices to prove the formula \eqref{eq:G^S}.
Let $P_{M,N}:=P\cap \big(\Z_M\tau+\Z_N \big)$, where $P= P_0 \cup P_1 $.
We now divide the sum on the right-hand side of \eqref{eq:G^S} into partial sums, corresponding to the index $j$ where the jump $\lambda_j\prec \bigl(\infty\tau+\infty=-\infty \tau-\infty \bigr) \prec \lambda_{j+1}$ occurs, as 
\begin{align*} 
\sum_{\substack{\lambda_1\prec \cdots \prec \lambda_d\\ \lambda_1,\ldots,\lambda_d \in (\Z_M\tau+\Z_N) \setminus\{0\}}} \frac{1}{\lambda_1^{k_1}\cdots \lambda_d^{k_d}} &= \sum_{j=0}^d \sum_{\substack{\lambda_1\prec \cdots \prec \lambda_d\\ \lambda_1,\ldots,\lambda_j\in P_{M,N}\\-\lambda_{j+1},\ldots,-\lambda_d\in P_{M,N}}} \frac{1}{\lambda_1^{k_1}\cdots \lambda_d^{k_d}}\\
&= \sum_{j=0}^d\sum_{\substack{\lambda_1< \cdots < \lambda_j\\ \lambda_1,\ldots,\lambda_j\in P_{M,N}}} \frac{1}{\lambda_1^{k_1}\cdots \lambda_d^{k_j}} \sum_{\substack{-\lambda_{j+1}< \cdots < -\lambda_d\\ \lambda_{j+1},\ldots,\lambda_d\in P_{M,N}}} \frac{1}{\lambda_{j+1}^{k_{j+1}}\cdots \lambda_d^{k_d}}.
\end{align*}
Thus, for $k_1,\ldots,k_d\ge2$, one gets
\[\lim_{M\rightarrow\infty}\lim_{N\rightarrow \infty} \sum_{\substack{\lambda_1\prec \cdots \prec \lambda_d\\ \lambda_1,\ldots,\lambda_d \in (\Z_M\tau+\Z_N) \setminus\{0\}}} \frac{1}{\lambda_1^{k_1}\cdots \lambda_d^{k_d}}= \sum_{j=0}^d (-1)^{k_{j+1}+\cdots+k_d} G_{k_1,\ldots,k_j}(\tau)G_{k_d,\ldots,k_{j+1}}(\tau).\]
By Theorem \ref{thm:G^sh}, the right-hand side of the above equation coincides with $G^{\shuffle,S}_{k_1,\ldots,k_d}(\tau)$.
Thus, the desired equation \eqref{eq:G^S} follows. Note that the assumption $k_1,\ldots,k_d\geq 2$ is imposed for conditional convergence of \eqref{eq:G^S}.
\end{proof}

As an analogue to Proposition \ref{prop:sym_MZV_gen}, we can rewrite the symmetric multiple Eisenstein series in terms of the power series $\Gamma_\mathrm{ME}$ defined in \eqref{eq:Gamma_MEMD}.

\begin{proposition}\label{prop:sym_MES_gen}
For positive integers $k_1,\ldots,k_d$, we have
\begin{align*}
 G^{\shuffle,S}_{k_1,\ldots,k_d}&=\langle \Gamma_\mathrm{ME} X_1 \Gamma_\mathrm{ME}^{-1}\mid e_1e_0^{k_1-1}\cdots e_1e_0^{k_d-1}e_1\rangle .
\end{align*}
\end{proposition}
\begin{proof}
The result is shown in much the same way as Proposition~\ref{prop:sym_MZV_gen}.
\end{proof}

We now prove that the symmetric multiple Eisenstein series satisfies the linear shuffle relation.
 
\begin{proposition}\label{prop:shuffle_linear_smes}
Let $d\ge1$.
For positive integers $k_1,\ldots,k_d$ and $j=0,1,\ldots,d-1$, we have
\[ \mathsf{G}^{\shuffle,S} \big(e_{k_1,\ldots,k_j} \shuffle e_{k_{j+1},\ldots,k_d}\big)=   (-1)^{k_{j+1}+\cdots+k_d}\mathsf{G}^{\shuffle,S} \big(e_{k_1,\ldots,k_j,k_{d},\ldots,k_{j+1}}\big). \]
\end{proposition}
\begin{proof}
Since $\Gamma_{\rm ME}$ is group-like, this is a direct consequence of Lemma \ref{lem:linear_shuffle_general}.
\end{proof}

The linear shuffle relation for the case $j=0$ is called the \emph{reversal relation}:
\begin{equation}\label{eq:reversal}
G^{\shuffle,S}_{k_1,\ldots,k_{d}}= (-1)^{k_{1}+\cdots+k_d} G^{\shuffle,S}_{k_{d},\ldots,k_{1}} \qquad (k_1,\ldots,k_d\geq 1).
\end{equation}
From this, we obtain $G^{\shuffle,S}_{\bk}=0$ if $\mathrm{wt}(\bk)$ is odd and the index $\bk$ is symmetric, that is, $\bk=(k_1,\ldots,k_d)=(k_d,\ldots,k_1)$ holds.

\subsection{Numerical observations} \label{ssc:num_obs}

Using Proposition \ref{prop:shuffle_linear_smes} and a computer (Wolfram Mathematica), we can obtain an upper bound of the dimension of the $\Q$-vector space $\mathcal{E}_k^S$ spanned by all $G_{\bk}^{\shuffle,S}$ of weight $k$. The results are listed in Table \ref{table:upper_bound}.
\begin{table}[H]
\caption{Upper bounds obtained from the linear shuffle relation}
\label{table:upper_bound}
\vspace*{0.7em}
\begin{tabular}{c|ccccccccccccc}
$k$ & 2 & 3 & 4 & 5 & 6 & 7 & 8 & 9 & 10&11\\ \hline
$\dim_\Q \mathcal{E}_k^S \le $ & 1 & 1 & 3 & 3 & 9 & 12 & 26 & 43 & 87 & 149 &  \\\hline
\end{tabular}
\end{table}

After reducing the number of generators by the linear shuffle relation as in Table \ref{table:upper_bound}, we could compute the numerical dimension of the space $\mathcal{E}_k^S$ up to weight 9.
Results are listed below (Table \ref{table:numerical_dim}).
\begin{table}[H]
\caption{Numerical dimensions}
\label{table:numerical_dim}
\vspace*{0.7em}
\begin{tabular}{c|ccccccccccccc}
$k$ & 2 & 3 & 4 & 5 & 6 & 7 & 8 & 9   \\ \hline
$\dim_\Q\mathcal{E}^S_k$ & 1 & 1 & 3 & 3 & 8 & 12 & 23 & 41 \\\hline
\end{tabular}
\end{table}

Comparing this with Table \ref{table:dim_E_k}, we observe that the space $\mathcal{E}_k^{S}$ is a proper subspace of the $\Q$-vector space spanned by all shuffle regularized multiple Eisenstein series of weight $k$.
Namely, not every $G_{\bk}^\shuffle$ can be expressed as a linear combination of symmetric multiple Eisenstein series, contrary to Yasuda's result \cite[Theorem 6.1]{Yasuda} for symmetric multiple zeta values.

\section{Linear shuffle space}
After reviewing the power series expression of the shuffle relation in Section~\ref{ssc:ps_shuffle}, we introduce the \emph{linear shuffle space} ${\rm LSh}^{(d)}_w$ in Section~\ref{ssc:lin_shuffle}, as the space consisting of rational polynomials in $d$ variables of homogeneous degree $w$ which are solutions to the linear shuffle equations.
We show that the dimension of the $\Q$-vector space $\mathcal{E}_{k}^{S,(d)}$ spanned by symmetric multiple Eisenstein series of weight $k$ and depth $d$ is bounded above by that of ${\rm LSh}^{(d)}_{k-d}$ (Proposition~\ref{prop:bound_E}).

\subsection{Power series expression of the shuffle relation} \label{ssc:ps_shuffle}
We rewrite the linear shuffle relation (Proposition \ref{prop:shuffle_linear_smes}) of depth $d$ in terms of polynomials in $d$ indeterminants. 
For this, let us first recall several standard notation for the expression of the shuffle relation in terms of commutative power series. Refer also to \cite[Section 8]{IKZ06}.

Let $R$ be a $\Q$-algebra and $R\llbracket x_1,\ldots,x_d\rrbracket $ the formal power series ring in $d$ indeterminants over $R$.
For $f^{(d)}\in R\llbracket x_1,\ldots,x_d\rrbracket $, we define the {\em $\sharp$-operator} by
\[ f^{(d)}(x_1,\ldots,x_d)^\sharp=f^{(d)}(x_1,x_1+x_2,\ldots,x_1+\cdots+x_d). \]
For $d\ge2$ and $j=1,2,\ldots,d-1$, the {\em $j$-th shuffle operator} $\mathrm{sh}_j^{(d)}=\mathrm{sh}_j$ is defined by 
\[ f^{(d)}(x_1,\ldots,x_d) \big| \mathrm{sh}_j := \sum_{\substack{\sigma\in\mathfrak{S}_d\\ \sigma(1)<\cdots<\sigma(j)\\\sigma(j+1)<\cdots<\sigma(d)}} f^{(d)}(x_{\sigma^{-1}(1)},\ldots,x_{\sigma^{-1}(d)}),\]
where $\mathfrak{S}_d$ is the symmetric group of degree $d$, viewed as the set of bijections on $\{1,2,\ldots,d\}$.

Now suppose that a family of elements $A_{\bk}\in R$ satisfy the shuffle relation; namely the $\Q$-linear map $\alpha \colon \mathfrak{H}_\shuffle^1\rightarrow R$ defined for each index $\bk$ by $\alpha(e_{\bk})=A_{\bk}$ is an algebra homomorphism.
Then, for $d\ge2$ and $j=1,2,\ldots, d-1$, the generating function $f^{(d)}(x_1,\ldots,x_d)=\sum_{k_1,\ldots,k_d\ge1} A_{k_1,\ldots,k_d}x_1^{k_1-1}\cdots x_d^{k_d-1}$ satisfies
\[
f^{(j)} (x_1,\ldots,x_j)^\sharp  f^{(d-j)} (x_{j+1},\ldots,x_d)^\sharp = \big(f^{(d)} (x_1,\ldots,x_d)\big|\mathrm{sh}_j\big)^\sharp.
\]
We let $\flat$ denote the inverse of the $\sharp$-operator, which is explicitly given by 
\[ f^{(d)}(x_1,\ldots,x_d)^{\flat}= f^{(d)}(x_1,x_2-x_1,\ldots,x_d-x_{d-1}). \]
Then,  by letting
\[ f^{(d)} (x_1,\ldots,x_d)^{\flat_j}:=f^{(d)}(x_1,x_2-x_1,\ldots,x_j-x_{j-1},x_{j+1},x_{j+2}-x_{j+1},\ldots,x_{d}-x_{d-1})\]
for $j=1,2,\ldots,d-1$, we get
\begin{align} \label{eq:shuffle_sharp_flat}
 f^{(j)} (x_1,\ldots,x_j)  f^{(d-j)} (x_{j+1},\ldots,x_d) = \big(\big(f^{(d)} (x_1,\ldots,x_d)\big|\mathrm{sh}_j\big)^\sharp\big)^{\flat_j}.
\end{align}
Comparing the coefficients, we obtain $A_{k_1,\ldots,k_j}A_{k_{j+1},\ldots,k_d}=\alpha(e_{k_1,\ldots,k_j}\shuffle e_{k_{j+1},\ldots,k_d})$.
For simplicity, we write $f^{(d)} (x_1,\ldots,x_d)\big|\mathrm{sh}_j'$ for the right-hand side of \eqref{eq:shuffle_sharp_flat}. Note that the polynomial ring $R[x_1,x_2,\ldots,x_d]$ is an $R$-subalgebra stable under all the operations $\sharp$, $\flat$, $\flat_j$, $\mathrm{sh}_j$ and $\mathrm{sh}_j'$.

Let us illustrate examples.
The operator $\mathrm{sh}_1'$ on $R\llbracket x_1,x_2\rrbracket $ is calculated as
\begin{equation*}\label{eq:LSh^2}
\begin{aligned}
f^{(2)}(x_1,x_2)\big|\mathrm{sh}_1' &=\big(\big( f^{(2)}(x_1,x_2)\big|\mathrm{sh}_1\big)^\sharp\big)^{\flat_1}
=\big(\big( f^{(2)}(x_1,x_2)+f^{(2)}(x_2,x_1)\big)^\sharp \big)^{ \flat_1} \\
&=\big( f^{(2)}(x_1,x_1+x_2)+f^{(2)}(x_2,x_1+x_2)\big)^{ \flat_1}\\
&= f^{(2)}(x_1,x_1+x_2)+f^{(2)}(x_2,x_1+x_2).
\end{aligned}
\end{equation*}
By a similar calculation, for the case $d=3$, we have
\begin{align*}
f^{(3)}(x_1,x_2,x_3)\big| \mathrm{sh}_1'
&=f^{(3)} (x_1,x_{12},x_{13})+f^{(3)}(x_2,x_{12},x_{13})+f^{(3)}(x_2,x_3,x_{13}),\\
f^{(3)}(x_1,x_2,x_3)\big| \mathrm{sh}_2'&=f^{(3)} (x_1,x_2,x_{23})+f^{(3)}(x_1,x_{13},x_{23})+f^{(3)}(x_3,x_{13},x_{23})  ,
\end{align*}
where we set $x_{ij}=x_i+x_j$.

\subsection{Linear shuffle relation revisited} \label{ssc:lin_shuffle}

We introduce the linear shuffle space $ {\rm LSh}^{(d)}_w$ of degree $w$ as a subspace of $\mathbb{Q}[x_1,\ldots,x_d]$ and show that its dimension gives an upper bound of the dimension of the space of symmetric multiple Eisenstein series.

For $j=0,1,\ldots,d-1$, we set
\[ f(x_1,\ldots,x_d)\big| p_j := (-1)^{d-j} f(x_1,\ldots,x_j,-x_d,\ldots,-x_{j+1}).\]
For positive integers $w$ and $d$, denote by $V^{(d)}_{w}$ the subspace of $\Q[x_1,\ldots,x_d]$ consisting of homogeneous polynomials of degree $w$.
We set $V_0^{(d)}:=\Q$.

\begin{definition} \label{def:lin_sh_space}
 For integers $w\ge0$ and $d\ge1$, the \emph{linear shuffle space} ${\rm LSh}^{(d)}_w$ of degree $w$ and depth $d$ is defined by 
\begin{align*} \label{eq:LSh_def}
 {\rm LSh}^{(d)}_w:=\{ f\in V^{(d)}_{w}\mid f \big|\mathrm{sh}_j' =f\big|p_j \ (j=1,2,\ldots,d-1),\ f=f\big| p_0\}.
\end{align*}
\end{definition}

For $d=1$, we have $ {\rm LSh}^{(1)}_w=\{f\in V^{(1)}_w\mid f(x_1)=-f(-x_1)\}$, which equals $\Q x_1^w$ if $w$ is odd and 0 otherwise.
Note that $ {\rm LSh}^{(d)}_0=0$.

For positive integers $k$ and $d$ with $k\ge d$, let $\mathcal{E}^{S,( d)}_k$ denote the $\Q$-vector space spanned by all symmetric multiple Eisenstein series of weight $k$ and depth $d$:
\begin{equation}\label{eq:E_k^S,d}
\mathcal{E}^{S,( d)}_k :=\langle G_{\bk}^{\shuffle,S} \mid \bk\in \N^d, \, \wt (\bk)=k\rangle_\Q.
\end{equation}

\begin{proposition}\label{prop:bound_E}
For integers $k\ge d\ge1$, the dimension of the space $\mathcal{E}^{S, (d)}_k$ is bounded above by that of $ {\rm LSh}^{(d)}_{k-d}:$
\begin{align*} \label{eq:upper_bound}
 \dim_{\mathbb{Q}} \mathcal{E}_k^{S,( d)} \le \dim_{\mathbb{Q}} {\rm LSh}^{(d)}_{k-d}.
\end{align*}
\end{proposition}
\begin{proof}
This follows from the same argument with \cite[Corollary~7]{IKZ06}.
For $k\ge d\ge1$, let 
\[ F_k^{(d)} (x_1,\ldots,x_d):= \sum_{\substack{k_1,\ldots,k_d\ge1\\k_1+\cdots+k_d=k}} G_{k_1,\ldots,k_d}^{\shuffle,S} \, x_1^{k_1-1}\cdots x_d^{k_d-1} \quad  \in \mathbb{C}\llbracket q\rrbracket [x_1,\ldots,x_d] \]
be the generating polynomial of the symmetric multiple Eisenstein series of weight $k$ and depth $d$, and $\{f_1,\ldots,f_s\}$ a basis of ${\rm LSh}_{k-d}^{(d)}$, where $s=\dim_\Q {\rm LSh}_{k-d}^{(d)}$.
It follows from the reversal relation \eqref{eq:reversal} that $F_k^{(d)}=F_k^{(d)}\big|p_0$.
From Proposition \ref{prop:shuffle_linear_smes} for the cases $j=1,2,\ldots,d-1$ and \eqref{eq:shuffle_sharp_flat}, the equalities
\[F_k^{(d)}\big|\mathrm{sh}_j' =F_k^{(d)} \big|p_j \quad (j=1,2,\ldots,d-1)\]
hold for all $k\ge d\ge 1$. Hence $F_k^{(d)}$ is contained in ${\rm LSh}_{k-d}^{(d)}\otimes_\mathbb{Q}\C\llbracket q\rrbracket $.
Using the fixed basis $f_1,\ldots,f_s$ of ${\rm LSh}_{k-d}^{(d)}$, we have
\[ F_k^{(d)}= \alpha_1f_1+\cdots+\alpha_s f_s\]
for some $\alpha_1,\ldots,\alpha_s\in \C\llbracket q\rrbracket$.
Since every $G^{\shuffle,S}_{\bk}$ with $\mathrm{wt}(\bk)=k$ and $\mathrm{dep}(\bk)=d$ is obtained as an appropriate higher derivative of $F^{(d)}_k$ with respect to $x_1,\ldots, x_d$, 
this shows that the space $\mathcal{E}^{S, (d)}_k$ is contained in a $\mathbb{Q}$-subspace of $\mathbb{C}\llbracket q\rrbracket$ spanned by $\alpha_1,\ldots,\alpha_s$, from which the desired result follows.
\end{proof}

This paper does not pursue a systematic study of the space ${\rm LSh}_w^{(d)}$. 
In Table \ref{table:dim_LSh_k}, we provide the list of the dimension of the space ${\rm LSh}_w^{(d)}$ obtained by computer (Wolfram Mathematica).

\begin{table}[H] 
\caption{Dimensions of ${\rm LSh}_w^{(d)}$}
\vspace*{0.8em}
\label{table:dim_LSh_k}
\begin{tabular}{c|ccccccccccccccccccccccccccccccccccc} \hline
$k$ & 1 & 2 & 3 & 4 & 5 & 6 & 7 & 8 & 9 & 10 & 11 & 12 & 13 & 14 & 15 & 16 \\\hline
$\dim {\rm LSh}_{k-1}^{(1)}$ & 0& 1 & 0 & 1 & 0 & 1 & 0 & 1 & 0 & 1 & 0 & 1 & 0 & 1 & 0 & 1 \\[.1em]\hline
$\dim {\rm LSh}_{k-2}^{(2)}$ & & 0 & 1 & 1 & 1 & 2 & 2 & 2 & 3 & 3 & 3 & 4 & 4 & 4 & 5 & 5 \\[.1em]\hline
$\dim {\rm LSh}_{k-3}^{(3)}$ & &  & 0 & 1 & 1 & 3 & 3 & 6 & 6 & 10 &
   10 & 15 & 15 & 21 & 21 & 28 \\[.1em]\hline
 $\dim {\rm LSh}_{k-4}^{(4)}$ & &  & & 0 &  1 & 2 & 4 & 7 & 11 & 17 & 24 & 33 & 44 & 57 & 73 & 91 
\end{tabular}
\end{table}


In the next section, we give the dimension formula and prove the equality in Proposition \ref{prop:bound_E} for $d=2$ when $k$ is odd; see Theorems \ref{thm:main} and \ref{thm:dim_LSh2}.

\section{$\mathfrak{S}_3$-representations, Fay-shuffle space and period polynomials}
We elucidate the linear shuffle space ${\rm LSh}^{(2)}_w$ of depth $2$ in this section.
After preparing several materials from representation theory in Section \ref{ssc:representation}, we give an alternative definition of ${\rm LSh}^{(2)}_w$ via $\mathfrak{S}_3$-representations, and compute its dimension in Section \ref{ssc:dim_LSh}.
As a by-product of this representation theoretical setup, we see that ${\rm LSh}^{(2)}_w$ is dual to the polynomial part of the {\em Fay-shuffle space} ${\rm FSh}^{\mathrm{pol}}_w$ when $w$ is odd (Theorem \ref{thm:lin_sh_vs_fay}). 
We also give a proof of Theorem \ref{thm:main2} in Section \ref{ssc:per_Fay}, which provides more direct connections with odd period polynomials and the Fay-shuffle space.

\subsection{Representation theoretical setups} \label{ssc:representation}

In what follows, we write 
\[V_w=\bigoplus_{r=0}^w \Q X^r Y^{w-r}\]
for the space of homogeneous polynomials of degree $w$ in two indeterminants $X$ and $Y$, instead of $V^{(2)}_w$ by abbreviation. By definition, a polynomial $P(X,Y)\in V_w$ lies in ${\rm LSh}^{(2)}_w$ if and only if it satisfies the following two equalities:
\begin{equation}\label{eq:LSh2}
\begin{aligned}
& P(X,X+Y)+P(Y,X+Y)=-P(X,-Y),\\
& P(X,Y)=P(-Y,-X).
\end{aligned}
\end{equation}

We first rewrite \eqref{eq:LSh2} in terms of the right group action of ${\rm GL}_2(\Z)$ on the polynomial space $V_w$, which is defined by
\begin{align*}
 (P\big|A)(X,Y)=P\bigl((X,Y) \, {}^t\! A\bigr)=P(aX+bY,cX+dY)
\end{align*}
for $A=\big(\begin{smallmatrix}a&b\\c&d\end{smallmatrix}\big)\in {\rm GL}_2(\Z)$ and $P(X,Y)\in V_w$. It is extended to the action of the group ring $\Q[{\rm GL}_2(\Z)]$ by linearity. Now consider the following three elements of $ {\rm GL}_2(\Z)$:
\[ \varepsilon=\begin{pmatrix}0&1\\1&0\end{pmatrix},\quad \gamma=\begin{pmatrix}0&-1\\1&-1\end{pmatrix},\quad \delta=\begin{pmatrix}1&0\\0&-1\end{pmatrix}.\]
Then one readily checks that $\varepsilon^2=\gamma^3=\delta^2=1$ holds, where the identity matrix in ${\rm GL}_2(\Z)$ is denoted by $1$ for simplicity.
Furthermore  $\varepsilon\gamma\varepsilon=\gamma^2=\gamma^{-1}$ holds, and thus the subgroup $G:=\langle \varepsilon ,\gamma\rangle $ generated by $\varepsilon$ and $\gamma$ is isomorphic to the symmetric group $\mathfrak{S}_3$ of degree 3.
We also let
\[ \varepsilon'=\delta\varepsilon\delta=\begin{pmatrix}0&-1\\-1&0\end{pmatrix},\quad \gamma'=\delta\gamma\delta=\begin{pmatrix}0&1\\-1&-1\end{pmatrix}.\]
Then the subgroup $G':=\langle \varepsilon',\gamma'\rangle$ generated by $\varepsilon'$ and $\gamma'$ is the conjugate of $G$ with respect to $\delta$, which is also isomorphic to $\mathfrak{S}_3$.

\begin{lemma}\label{lem:LSh_2_S_3}
For a positive integer $w$, we have
\[{\rm LSh}^{(2)}_w= \{P\in V_w\mid P\big|(1-\varepsilon')=P\big|(1+\gamma+\gamma^2)=0\}.\]
\end{lemma}
\begin{proof}
Observe that $\gamma\delta= \big(\begin{smallmatrix} 0&1\\1&1\end{smallmatrix}\big)$ and $ \varepsilon'\gamma^2\delta = \big(\begin{smallmatrix} 1&0\\1&1\end{smallmatrix}\big)$ hold.
Therefore, the defining equations \eqref{eq:LSh2} can be rewritten as
\[P\big|(\gamma\delta+\varepsilon'\gamma^2\delta+\delta)=P\big|(1-\varepsilon')=0.\]
For $P\in V_w$ satisfying $P\big|(1-\varepsilon')=0$, the condition $P\big|(\gamma\delta+\varepsilon'\gamma^2\delta+\delta)=0$ is equivalent to $P\big|(1+\gamma+\gamma^2)=0$, because we have $P\big|\varepsilon'\gamma^2\delta=P\big| \gamma^2\delta$ due to $P\big|\varepsilon'=P$.
This completes the proof.
\end{proof}

As explained at \cite[Section 3]{GKZ}, there exists a standard ${\rm GL}_2(\Z)$-equivariant non-degenerate pairing $\langle\, ,\,\rangle\colon V_w\otimes_\mathbb{Q} V_w\rightarrow \Q(\det^w)$  defined by
\[ \langle X^{a}Y^{w-a},X^{b}Y^{w-b}\rangle = (-1)^{a+1}\binom{w}{b}^{-1} \delta_{a,w-b}\] 
for $0\leq a,b\leq w$, where $\delta_{a,w-b}$ denotes the Kronecker delta. Here $\mathbb{Q}(\det^w)$ denotes $\mathbb{Q}$ on which $A\in {\rm GL}_2(\Z)$ acts via multiplication of $(\det A)^w$, and  the $\mathrm{GL}_2(\Z)$-equivariance stands for the equality $ \langle P\big|A,Q\big|A\rangle = (\det A)^w \langle P,Q\rangle$ for $A\in {\rm GL}_2(\Z)$ and $P,Q\in V_w$.
Now let us define another pairing $(\,,\,)\colon V_w\otimes_\mathbb{Q}V_w\rightarrow \mathbb{Q}$ by setting
\begin{equation}\label{eq:def_pairing}
 (P,Q):= \langle P\big|\delta, Q\rangle
 \end{equation}
for $P,Q\in V_w$. Then, for any $g\in G$ and $P,Q\in V_w$, we have
\begin{equation}\label{eq:(,)} 
\big(P\big|g,Q\big|g'\big)=(\det g)^w \big(P,Q\big)
\end{equation}
where $g'\in G'$ is defined as $g'=\delta g\delta$ as before. 

Recall that $\mathfrak{S}_3$ is the symmetric group of degree $3$, which is regarded as the permutation group of three letters $1,2,3$. It is known to be generated by a transposition $(1\; 2)$ of $1$ and $2$ and a cyclic permutation $(1\; 2\; 3)$ of $1$, $2$ and $3$. Note that, since $G=\langle \varepsilon,\gamma\rangle$ and $G'=\langle \varepsilon',\gamma'\rangle$ are subgroups of $\mathrm{GL}_2(\Z)$ isomorphic to $\mathfrak{S}_3$, we may regard $V_w$ as a right $\mathfrak{S}_3$-module in two different ways. 

\begin{convention} \label{conv:V}
We let $V_w$ denote a right $\mathfrak{S}_3$-module $V_w$ on which $\mathfrak{S}_3$ acts via 
\begin{align*}
 \mathfrak{S}_3 \hookrightarrow \mathrm{GL}_2(\Z); \, (1\;2) \mapsto \varepsilon, (1\;2\;3) \mapsto \gamma.
\end{align*}
Meanwhile, we let $V_w'$ denote a right $\mathfrak{S}_3$-module  $V_w$ on which $\mathfrak{S}_3$ acts via 
\begin{align*}
 \mathfrak{S}_3  \hookrightarrow  \mathrm{GL}_2(\Z); \, (1\;2) \mapsto \varepsilon', (1\;2\;3) \mapsto \gamma'.
\end{align*}
For $\tau\in \mathfrak{S}_3$, we respectively write $g_\tau\in G$ and $g_\tau'\in G'$ for the corresponding element via the former and the latter one. For example, we have $\varepsilon'=g_{(1\;2)}'$. 
\end{convention}

Under Convention~\ref{conv:V}, the equality \eqref{eq:(,)} is interpreted as the $\mathfrak{S}_3$-equivariance of the pairing $(\,,\,)\colon V_w\otimes_{\mathbb{Q}}V_w'\rightarrow \mathbb{Q}(\mathrm{sgn}^w)$, where $\mathbb{Q}(\mathrm{sgn}^w)$ denotes $\mathbb{Q}$ on which $\sigma\in \mathfrak{S}_3$ acts via $(\mathrm{sgn}\, \sigma)^w$ with the signature function $\mathrm{sgn}\colon \mathfrak{S}_3\rightarrow\{\pm 1\}$.

\subsection{Dimension of the linear shuffle space ${\rm LSh}^{(2)}_w$} \label{ssc:dim_LSh}

We compute the dimension of ${\rm LSh}_w^{(2)}$ by relating this with the subspace $W_w$ of $V_w' \ (=V_w)$ defined as
\begin{equation}\label{eq:W_w}
W_w:=\{Q\in V_w'\mid Q\big|(1-\varepsilon')=Q\big|(1+\gamma'+{\gamma'}^2)=0\}.
\end{equation}
The space $W_w$ was introduced in the proof of \cite[Lemma 3.13]{Matthes} in order to compute the dimension of the space of {\em elliptic double zeta values}.
Its dimension for $w\geq 0$ is given by
\begin{equation}\label{eq:dim_W_w}
\dim_{\mathbb{Q}} W_w = \left\lfloor\frac{w+2}{3}\right\rfloor.
\end{equation}
The proof of \eqref{eq:dim_W_w} is based on the $\mathfrak{S}_3$-representation theory; see \cite[Proposition 3.14]{Matthes}.
We will also use the $\mathfrak{S}_3$-representation theory to prove the following result.

\begin{theorem}\label{thm:dim_LSh2}
Let $w$ be a non-negative integer.
If $w$ is even, the $\mathbb{Q}$-linear map
\[ {\rm LSh}_w^{(2)}\longrightarrow W_w; \quad P\longmapsto P\big|\delta\]
is well-defined and bijective.
If $w$ is odd, the pairing $\eqref{eq:def_pairing}$ induces an isomorphism 
\[{\rm LSh}^{(2)}_w\cong {\rm Hom}_{\mathbb{Q}}(W_w,\Q)\]
of $\Q$-vector spaces, where ${\rm Hom}_{\mathbb{Q}}(W_w,\Q)$ is the space of $\Q$-linear maps from $W_w$ to $\Q$. 
In both cases, the dimension of ${\rm LSh}^{(2)}_w$ is determined as 
\[\dim_{\Q} {\rm LSh}^{(2)}_w=\left\lfloor\frac{w+2}{3}\right\rfloor.\]
\end{theorem}
\begin{proof}
First, assume that $w$ is even. Then we have $P\big|\varepsilon=P\big|\varepsilon'$ for every $P\in V_w$.
Now take $P\in {\rm LSh}^{(2)}_w$ and put $Q=P\big|\delta$.
Then $Q\big|(1+\gamma'+{\gamma'}^2)=P\big|(1+\gamma+\gamma^2)\delta=0$, which implies $Q=P\big|\delta \in W_w$. The inverse map is obviously given by $Q\mapsto Q\big|\delta$, and thus the map under consideration is bijective as desired.

\

Now suppose that $w$ is odd.
We use the representation theory of $\mathfrak{S}_3$ together with the pairing $(\,,\,) \colon V_w\otimes V_w'\rightarrow \Q(\mathrm{sgn}^w)$ defined in \eqref{eq:def_pairing}, which is $\mathfrak{S}_3$-equivariant under Convention~\ref{conv:V}.  Since it is non-degenerate, gives rise to the $\mathfrak{S}_3$-equivariant isomorphism
\[ \psi\colon V_w\longrightarrow {\rm Hom}_{\mathbb{Q}}(V_w',\Q(\mathrm{sgn}^w));\quad P\longmapsto f_P, \]
where the $\mathbb{Q}$-linear map $f_P\in {\rm Hom}_{\mathbb{Q}}(V_w',\Q(\mathrm{sgn}^w))$ is defined by $f_P(Q)=(P,Q)$ for $Q\in V_w'$.
Namely, we have
\[ V_w^{\prime \, \vee}(\mathrm{sgn}^w):={\rm Hom}_{\mathbb{Q}}(V_w',\Q(\mathrm{sgn}^w))=\{f_P :V_w'\rightarrow \Q(\mathrm{sgn}^w)\mid P\in V_w\}.\] 
Recall that the contragredient $\mathfrak{S}_3$-action on $V_w^{\prime\, \vee}(\mathrm{sgn}^w)$ is defined by 
\[ \big(\tau f\big) (Q):=(\mathrm{sgn} \,\tau)^w f(Q\big|(g_\tau')^{-1}) \]
for $Q\in V'_w$, $\tau \in \mathfrak{S}_3$ and $f\in V_w^{\prime\, \vee}(\mathrm{sgn}^w)$.
The $\mathfrak{S}_3$-equivariance of $\psi$ is a standard fact of representation theory, which can be readily checked as 
\begin{align*}
 \psi(P\big|g_\tau)(Q)&=(P\big|g_\tau,Q) =(\mathrm{sgn}\, \tau)^w (P,Q\big|(g_\tau')^{-1})=(\mathrm{sgn}\, \tau)^w \psi(P)(Q\big|(g_\tau')^{-1})
\end{align*}
for $\tau \in \mathfrak{S}_3$, $P\in V_w$ and $Q\in V_w'$; here the second equality follows from \eqref{eq:(,)}.  

For $\mathbb{Q}$-representations $\varrho$ and $V$ of $\mathfrak{S}_3$ with $\varrho$ irreducible, we write $V[\varrho]$ for the isotypic component of $V$ of type $\varrho$. One readily checks that the isotypic component satisfies $V^\vee[\varrho]\cong V[\varrho]^\vee$ and $(V\otimes \chi)[\varrho]\cong V[\varrho\otimes \chi^{-1}]\otimes \chi$, where the superscript $\vee$ denotes the contragredient representation and $\chi$ is any $1$-dimensional representation of $\mathfrak{S}_3$. Here we use the fact that any irreducible representation of $\mathfrak S_3$ is conjugate self-dual. Now take an irreducible $2$-dimensional $\mathbb{Q}$-representation $\rho$ of $\mathfrak{S}_3$. The map $\psi$ induces an isomorphism of the isotypic components of type $\rho$, and combining with the isomorphisms above, we have
\[V_w[\rho] \cong  V_w^{'\, \vee}(\mathrm{sgn}^w)[\rho]\cong V_w^{'\, \vee}[\rho \otimes \mathrm{sgn}^{-w}](\mathrm{sgn}^w)\cong  V_w'[\rho\otimes \mathrm{sgn}^{-w}]^\vee(\mathrm{sgn}^w). \]
However, since an irreducible $2$-dimensional $\mathbb{Q}$-representation of $\mathfrak{S}_3$ is known to be unique up to isomorphism, there exists an $\mathfrak{S}_3$-isomorphism $(\rho\otimes \mathrm{sgn}^{-w})^\vee\cong \rho$ and we thus have
\begin{equation*}
V_w[\rho]\cong V_w'[\rho]^{\vee}(\mathrm{sgn}^w)={\rm Hom}_{\mathbb{Q}}(V_w'[\rho],\mathbb{Q}(\mathrm{sgn}^w)).
\end{equation*}
For an $\mathfrak{S}_3$-representation $\mathcal{V}$ with coefficients in $\mathbb{Q}$, let $\mathcal{V}^{\pm}$ denote the eigenspace on which $(1\; 2)$ acts via multiplication by $\pm 1$.
Then, the above isomorphism preserves the $(\pm1)$-eigenspaces of $(1\; 2)\in \mathfrak{S}_3$.
Moreover, since $(1\; 2)$ acts on $\Q(\mathrm{sgn}^w)$ as the scalar $-1$ (recall that $w$ is assumed to be odd), we obtain
\begin{equation}\label{eq:rho}  
V_w[\rho]^\pm \cong V_w'[\rho]^{\mp}(\mathrm{sgn}^w)= {\rm Hom}_{\mathbb{Q}}(V_w'[\rho]^{\mp},\Q(\mathrm{sgn}^w)).
\end{equation}

By Maschke's theorem, there exists a unique $\mathfrak{S}_3$-subrepresentation $U_w$ of $V_w$ satisfying $V_w=U_w\oplus V_w[\rho]$. Due to the classification of irreducible representations of $\mathfrak{S}_3$, one observes that $U_w$ is the maximal $\mathbb{Q}[\mathfrak{S}_3]$-submodule on which $(1\;2\;3)\in \mathfrak{S}_3$ acts trivially. Since $0=(1\;2\;3)^3-1=\bigl((1\;2\;3)-1\bigr)\bigl((1\;2\;3)^2+(1\;2\;3)+1\bigr)$ holds in $\mathbb{Q}[\mathfrak{S}_3]$, we observe that the orthogonal compliment $V_w[\rho]$ of $U_w$ is then characterised as the maximal $\mathbb{Q}[\mathfrak{S}_3]$-submodule of $V_w$ annihilated by $(1\;2\;3)^2+(1\;2\;3)+1$, which implies
\begin{align*}
V_w[\rho]&=\{P\in V_w\mid P\big|(1+\gamma+\gamma^2)=0\}.
 \end{align*}
By the same argument on the decomposition $V_w'=U_w'\oplus V_w'[\rho]$, we also obtain
\begin{align*}
 V_w'[\rho]&=\{Q\in V_w'\mid Q\big|(1+\gamma'+{\gamma'}^2)=0\}.
\end{align*}
From Lemma \ref{lem:LSh_2_S_3} and the definition of $W_w$, we see that ${\rm LSh}^{(2)}_w=V_w[\rho]^-$ and $W_w=V_w'[\rho]^+$ hold. 
The desired isomorphism $\mathrm{LSh}^{(2)}_w\cong \mathrm{Hom}_\mathbb{Q}(W_w,\mathbb{Q})$ now directly follows from \eqref{eq:rho} combined with these identifications.

Since we have verified that $\mathrm{LSh}^{(2)}_w$ is isomorphic to $W_w$ as $\Q$-vector spaces, the desired formula for $\dim_{\mathbb{Q}}\mathrm{LSh}^{(2)}_w$ immediately follows from the dimension formula for $W_w$ given in \eqref{eq:dim_W_w}.
\end{proof}

\subsection{Relation with odd period polynomials and the Fay-shuffle space} \label{ssc:per_Fay}

We now relate the linear shuffle space ${\rm LSh}_w^{(2)}$ to the space $W_w^{\rm od}$ of odd period polynomials and to the polynomial subspace ${\rm FSh}_w^\mathrm{pol}$ of the Fay-shuffle space due to Matthes \cite{Matthes}.
Let us briefly recall the definition and the context of these two spaces.

For an integer $w\ge0$, let 
\[W_w^{\rm od}:=\{P\in V_w\mid P(X,Y)-P(X+Y,Y)-P(X+Y,X)=0\}\]
be the space of odd period polynomials of degree $w$.
This notation is conventional, and while there may be a concern about overlap with $W_w$ defined in \eqref{eq:W_w}, it is unlikely to cause any confusion.
When $w$ is even, the defining equation of odd period polynomials is originated from $\Q$-linear relations among critical values of $L$-functions of a cusp form.
Moreover, due to the Eichler--Shimura--Manin theory (see \cite{GKZ} for example), the space $W_w^{\rm od}\otimes \C$ is isomorphic to the $\C$-vector space $S_{w+2}({\rm SL}_2(\Z))$ of cusp forms of weight $w+2$ for ${\rm SL}_2(\Z)$.
From this, we have
\[ \dim_{\mathbb{Q}} W_w^{\rm od} = \dim_{\mathbb{C}} S_{w+2}({\rm SL}_2(\Z)) = \left\lfloor \frac{w-2}{4}\right\rfloor - \left\lfloor \frac{w}{6}\right\rfloor. \]
Note that $W_w^{\rm od}=0$ holds for $w$ odd.

For an integer $w\ge0$, the Fay-shuffle space ${\rm FSh}_w$, introduced by Matthes (denoted by ${\rm FSh}_2(w)$ in \cite[Definition 3.8]{Matthes}), consists of Laurent polynomials $Q\in V_{w}+\Q \frac{Y^{w+1}}{X}+\Q \frac{X^{w+1}}{Y}$ satisfying 
\[ Q(X,Y)+Q(Y,X)=Q(X,Y)+Q(X+Y,-Y)+Q(-X-Y,X)=0.\]
The space ${\rm FSh}_w$ is trivial when $w$ is even.
For $w$ odd, the defining equations of the space ${\rm FSh}_w$ arise from the \emph{Fay-shuffle relation} and the shuffle relation modulo products satisfied by elliptic double zeta values \cite[Proposition 2.5]{Matthes}.
Note that, as shown in \cite[Proposition 3.12]{Matthes}, the Fay-Shuffle space is decomposed into
\[ {\rm FSh}_w={\rm FSh}_w^\mathrm{pol} \oplus \Q\left(\frac{X^{w+1}}{Y}-\frac{Y^{w+1}}{X}-\frac{X^{w+1}-Y^{w+1}}{X+Y}\right),\] 
where $ {\rm FSh}_w^\mathrm{pol} ={\rm FSh}_w\cap V_w$ denotes the polynomial part of ${\rm FSh}_w$.
As pointed out in the proof of \cite[Lemma 3.13]{Matthes}, the space $ {\rm FSh}_w^\mathrm{pol}$ coincides with $W_w$ defined in \eqref{eq:W_w}: ${\rm FSh}_w^\mathrm{pol} = W_w$.
Thus, it follows from Theorem \ref{thm:dim_LSh2} that 
\[{\rm LSh}_w^{(2)}\cong {\rm Hom}_\Q ({\rm FSh}_w^\mathrm{pol},\Q).\]

We now give a precise formulation of Theorem \ref{thm:main2} and provide its proof.

\begin{theorem}\label{thm:lin_sh_vs_fay}
\begin{enumerate}[leftmargin=25pt,label={\rm (\roman*)}]
 \item For $w\ge0$ even, the space of odd period polynomials $W_w^{\rm od}$ is contained in ${\rm LSh}_w^{(2)}$.
\item For $w\ge 1$ odd, the $\Q$-linear map 
\[ {\rm FSh}_w^\mathrm{pol} \longrightarrow {\rm LSh}_w^{(2)};\quad Q\longmapsto Q\big|\gamma' (1-\varepsilon')\delta \]
is a well-defined bijection, whose inverse is given by
\[ {\rm LSh}_w^{(2)} \longrightarrow {\rm FSh}_w^\mathrm{pol};\quad P\longmapsto -\frac13 P\big|\gamma (1+\varepsilon)\delta. \] 
\end{enumerate}
\end{theorem}
\begin{proof}
(i)  
Recall that every odd period polynomial $P\in W_w^{\rm od}$ satisfies $P\big|(1+\delta)=0$ and $P\big|(1-\varepsilon)=0$; see \cite[Section 5]{GKZ}.
Since $w$ is even, we have $P\big|(1-\varepsilon')=0$.
From the assumption $P(X+Y,Y)+P(X+Y,X)\stackrel{(*)}{=}P(X,Y)$, we get
\begin{align*}
P(X,X+Y)+P(Y,X+Y)&=(P\big| \varepsilon)(X,X+Y)+(P\big|\varepsilon)(Y,X+Y)  \\
&=P(X+Y,X)+P(X+Y,Y)\stackrel{(*)}{=}P(X,Y)\\
&=(-P\big|\delta)(X,Y)=-P(X,-Y).
\end{align*}
Thus \eqref{eq:LSh2} implies that $P\in {\rm LSh}_w^{(2)}$.

\noindent (ii) Let $Q\in V_w$. Since $w$ is odd, we have
\begin{align*}
 Q(X,Y)+Q(Y,X)=Q(X,Y)-Q(-Y,-X)=Q(X,Y)-(Q\big| \varepsilon')(X,Y).
\end{align*}
This shows that $Q$ satisfies $Q(X,Y)+Q(Y,X)=0$ if and only if $Q\big|(1-\varepsilon')=0$ holds. We may also check the equalities
\begin{align*}
Q(X+Y,-Y) &=(Q\big| \varepsilon'\gamma')(X,Y)=(Q\big| \gamma')(X,Y),  & Q(-X-Y,X)&=(Q\big|\gamma^{\prime \, 2})(X,Y), 
\end{align*}
and thus $Q$ satisfies $Q(X,Y)+Q(X+Y,-Y)+Q(-X-Y,X)=0$ if and only if it satisfies $Q\big|(1+\gamma'+\gamma^{\prime\, 2})=0$.
Hence, the defining equations of the space ${\rm FSh}_w^\mathrm{pol}$ coincide with those of $W_w$ in \eqref{eq:W_w}; see also \cite[Section 3]{Matthes}.
Now let us assume $Q\in {\rm FSh}_w^\mathrm{pol}$. Since $(1-\varepsilon')(1+\varepsilon')=0$ in $\Q[{\rm GL}_2(\Z)]$, we have $Q\big|\gamma' (1-\varepsilon')\delta (1+\varepsilon)=Q\big|\gamma' (1-\varepsilon') (1+\varepsilon')\delta=0$.
Using $\varepsilon'(1+\gamma'+{\gamma'}^2)=(1+\gamma'+{\gamma'}^2)\varepsilon'$, one computes
\begin{align*}
Q\big|\gamma' (1-\varepsilon')\delta (1+\gamma+\gamma^2)&=Q\big|\gamma' (1-\varepsilon') (1+\gamma'+{\gamma'}^2)\delta\\
&=Q\big|(1+\gamma'+{\gamma'}^2)(1-\varepsilon')\delta=0.
\end{align*}
This implies $Q\big|\gamma' (1-\varepsilon')\delta\in {\rm LSh}_w^{(2)}$.
Similarly we can also check that $P\big|\gamma (1+\varepsilon)\delta\in {\rm FSh}_w^\mathrm{pol}$ for $P\in {\rm LSh}_w^{(2)}$.

One readily checks that the compositions of the former and latter maps are the identity maps.
Indeed, direct computations show that
\begin{align*}
-\frac13 P\big|\gamma (1+\varepsilon)\delta\gamma' (1-\varepsilon')\delta&=-\frac13 P\big| \gamma(1+\varepsilon)\gamma(1-\varepsilon)=-\frac13 P\big|(\gamma^2+\varepsilon)(1-\varepsilon)\\
&=-\frac13 P\big|(\gamma^2-\varepsilon\gamma+\varepsilon-1)=-\frac13 P\big|(\gamma^2+\gamma-2)=P, \\
-\dfrac{1}{3}Q\big|\gamma'(1-\varepsilon')\delta  \gamma (1+\varepsilon)\delta &=-\dfrac{1}{3}Q\big| \gamma'(1-\varepsilon')\gamma'(1+\varepsilon') =-\dfrac{1}{3}Q\big|(\gamma^{\prime\, 2}-\varepsilon')(1+\varepsilon') \\ 
 &=-\dfrac{1}{3}Q\big| (\gamma^{\prime \, 2}+\varepsilon'\gamma'-\varepsilon'-1)=-\dfrac{1}{3}Q\big|(\gamma^{\prime\, 2}+\gamma'-2)=Q
\end{align*}
for $P\in {\rm LSh}_w^{(2)}$ and $Q\in {\rm FSh}_w^\mathrm{pol}$. Therefore, these maps yield mutually inverse isomorphisms between
$\mathrm{FSh}_w^{\mathrm{pol}}$ and $\mathrm{LSh}_w^{(2)}$.
\end{proof}

We here give several examples. For $w$ even, the first example of an odd period polynomial is given by 
\[ s_{12}= 4X^9Y-25X^7Y^3+42X^5Y^5-25X^3Y^7+4XY^9 \in W_{10}^{\rm od}.\]
See \cite[Section 6]{Zagier} for more examples of odd period polynomials. Meanwhile, the following polynomials form a basis of ${\rm LSh}_{10}^{(2)}$:
\begin{align*}
P_1=&X^7 Y^3-4 X^6 Y^4+6 X^5 Y^5-4 X^4 Y^6+X^3 Y^7,\\
P_2=&X^8 Y^2-11 X^6 Y^4+20 X^5 Y^5-11 X^4 Y^6+X^2 Y^8,\\
P_3=&X^9 Y-25 X^6 Y^4+48 X^5 Y^5-25 X^4 Y^6+X Y^9,\\
P_4=&X^{10}-50 X^6 Y^4+96 X^5 Y^5-50 X^4 Y^6+Y^{10}.
\end{align*}
Theorem~\ref{thm:lin_sh_vs_fay} (i) implies that $s_{12}$ can be written as a $\mathbb{Q}$-linear combination of $P_1$, $P_2$, $P_3$ and $P_4$; indeed we have $s_{12}=4P_3-25P_1$.

For comparison, we provide bases of ${\rm FSh}_w^{\mathrm{pol}}$ and ${\rm LSh}_w^{(2)}$ for odd $w$ in Tables \ref{table:FSh_bases} and \ref{table:LSh_bases}, respectively.
\begin{table}[H]
\caption{Bases of $\mathrm{FSh}_w^{\mathrm{pol}}$ for odd $w$}
\label{table:FSh_bases}
\vspace*{.8em}
\begin{tabular}{l|ll} \hline
$w$ & a basis of ${\rm FSh}_w^{\mathrm{pol}}$\\ \hline
$1$ & $X-Y$\\\hline
$3$ & $X^3 - Y^3$\\\hline
$5$ & $X^5 - Y^5,\ X^4 Y + X^3 Y^2 - X^2 Y^3 - X Y^4$\\\hline
$7$ & $X^7 - Y^7,\ X^5 Y^2+X^4 Y^3-X^3 Y^4-X^2 Y^5,\ X^6 Y-X^4 Y^3+X^3 Y^4-X Y^6$\\\hline
\end{tabular}
\end{table}

\begin{table}[H]
\caption{Bases of $\mathrm{LSh}_w^{(2)}$ for odd $w$}
\label{table:LSh_bases}
\vspace*{.8em}
\begin{tabular}{l|ll} \hline
$w$ & a basis of ${\rm LSh}_w^{(2)}$\\ \hline
$1$ & $X-Y$\\\hline
$3$ & $X^3 - 2 X^2 Y + 2 X Y^2 - Y^3$\\\hline
$5$ & $X^4 Y - 5 X^3 Y^2 + 5 X^2 Y^3 - X Y^4,\ X^5 - 10 X^3 Y^2 + 
 10 X^2 Y^3 - Y^5$\\\hline
$7$ & $X^5 Y^2-3 X^4 Y^3+3 X^3 Y^4-X^2 Y^5,\ X^6 Y-7 X^4 Y^3+7 X^3 Y^4-X Y^6$,\\
&$X^7-14 X^4 Y^3+14 X^3 Y^4-Y^7$\\\hline
\end{tabular}
\end{table}

\section{Proof of the main results}
In this section, we prove Theorems  \ref{thm:modular_relations} and \ref{thm:main}. We first complete the proof of Theorem~\ref{thm:main} in Sections \ref{ssc:1.2-i} and \ref{ssc:1.2-ii}, and then finally verify Theorem~\ref{thm:modular_relations} in Section~\ref{ssc:1.1}.

\subsection{Proof of Theorem \ref{thm:main} (i)} \label{ssc:1.2-i}
Let $k\ge4$ be even. 
Note that $(2\pi i)^k \mathcal{DE}_k^S = \mathcal{E}_k^{S,(2)}$ holds, where the $\Q$-vector spaces $\mathcal{DE}_k^S$ and $\mathcal{E}_k^{S,(2)}$ are defined in \eqref{eq:DE_k^S} and \eqref{eq:E_k^S,d}, respectively.
Let $M_k^\Q({\rm SL}_2(\Z))$ denote the $\Q$-vector space spanned by modular forms with rational Fourier coefficients. 
In this subsection, we shall prove 
\[\mathcal{DE}_k^S= M_k^\Q({\rm SL}_2(\Z)) \oplus \Q  \widetilde{G}_{k-2}'(\tau).\]

The following fact is well known in the study of modular forms.

\begin{theorem}\label{thm:two_product} 
For $k\ge4$ even, every modular form of weight $k$ for ${\rm SL}_2(\Z)$ with rational Fourier coefficients is presented as a $\Q$-linear combination of $\widetilde{G}_k(\tau)$ and $ \widetilde{G}_{2j}(\tau)\widetilde{G}_{k-2j}(\tau)$ for $j=2,3,\ldots,\left\lfloor\frac{k}{4}\right\rfloor$.
\end{theorem}

For the basis theorems, see \cite[Theorem~1.1]{Fukuhara} and \cite[Corollary~1]{HNT}.

When $k$ is even, the symmetric double Eisenstein series of weight $k$ has the following expression.

\begin{proposition}\label{prop:Gsym_dep2_even}
Let $k\ge4$ even.
For $r,s\ge2$ with $k=r+s$, we have
\begin{equation*}\label{eq:dep2}
\begin{aligned}
G^{\shuffle,S}_{r,s} 
&=\begin{cases} 2G_{r}^\shuffle G_{s}^\shuffle-G_{k}^\shuffle & \text{\rm if }r,s:{\rm even},\\ -G_{k}^\shuffle & \text{\rm if }r,s:{\rm odd}.\end{cases} 
\end{aligned}
\end{equation*}
When $r=1$, we obtain 
\begin{align*}
 G_{1,k-1}^{\shuffle,S}&=\dfrac{(2\pi i)^2}{2(k-2)} q\frac{d}{dq}G_{k-2}^\shuffle -G_k^\shuffle.
\end{align*}
\end{proposition}

Note that $G_{1,k-1}^{\shuffle,S}=G_{k-1,1}^{\shuffle,S}$ holds for even $k$ due to the reversal relation \eqref{eq:reversal}.

\begin{proof}
For $r,s\ge2$, the stuffle relation for $G_{r}^\shuffle G_{s}^\shuffle$ implies 
\begin{align*}
G^{\shuffle,S}_{r,s}&= G_{r,s}^\shuffle + (-1)^{s}G_{r}^\shuffle G_{s}^\shuffle +G_{s,r}^\shuffle\\ 
& =G_{r,s}^\shuffle+(-1)^s \bigl(G_{r,s}^\shuffle+G_{s,r}^\shuffle+G_k^\shuffle\bigr)+G_{s,r}^\shuffle \\
&=(1+(-1)^{s}) G_{r}^\shuffle G_{s}^\shuffle-G_{k}^\shuffle,
\end{align*}
which is none other than the former identity.
For the latter one, we use the shuffle relation to obtain
\[ G_{1,k-1}^{\shuffle,S}=G_{1,k-1}^\shuffle - G_1^{\shuffle}G_{k-1}^\shuffle + G_{k-1,1}^\shuffle=-\sum_{j=1}^{k-2} G_{j,k-j}^{\shuffle}.\]
Then, the result follows from Kaneko's sum formula \eqref{eq:sum_formula}.
\end{proof}

\medskip
\begin{proof}[Proof of Theorem~$\ref{thm:main}$ {\rm (i)}]
For $k\ge6 $ even, Proposition \ref{prop:Gsym_dep2_even} shows that the space $\mathcal{DE}_k^{S}$ is generated by $\widetilde{G}_k^\shuffle,q\frac{d}{dq}\widetilde{G}_{k-2}^\shuffle$ and $\widetilde{G}_{2j}^\shuffle \widetilde{G}_{k-2j}^\shuffle$ for $j=1,2,\ldots,\left\lfloor\frac{k}{4}\right\rfloor$.
The desired result then follows from Theorem \ref{thm:two_product} and the dependence of $q\frac{d}{dq}\widetilde{G}_{k-2}^\shuffle$ and $\widetilde{G}_{2}^\shuffle \widetilde{G}_{k-2}^\shuffle$ modulo modular forms; in fact, $\widetilde{G}_2(\tau)\widetilde{G}_{k-2}(\tau) + \frac{\widetilde{G}'_{k-2}(\tau)}{2(k-2)}$ is contained in $M_k^\Q({\rm SL}_2(\Z))$ for $k\ge6 $ even.
\end{proof}

Remark that Theorem~$\ref{thm:main}$ {\rm (i)} does not hold for $k=4$.
This can be seen from the fact that $\dim_\Q \mathcal{DE}_4^S\le 1$ (see Table \ref{table:dim_LSh_k}).
More precisely, from Proposition \ref{prop:Gsym_dep2_even} together with the identity $2 \widetilde{G}_2 (\tau) \widetilde{G}_{2}(\tau) = 5\widetilde{G}_4(\tau)-\widetilde{G}_{2}'(\tau) $, we have $\widetilde{G}_{2,2}^{\shuffle,S}(\tau)=4\widetilde{G}_4(\tau)-\widetilde{G}_{2}'(\tau)$.
Combining this with $\widetilde{G}^{\shuffle,S}_{1,3}(\tau)=\frac{1}{4} \widetilde{G}_2'(\tau)-\widetilde{G}_4(\tau)$, we get $\dim_\Q \mathcal{DE}_4^S= 1$.

When $k$ equals $2$, we have $G_{1,1}^{\shuffle,S}=0$ from \eqref{eq:reversal}, 
and thus we get $ \mathcal{DE}_2^S=\Q \widetilde{G}_2$ by \eqref{eq:dep1}.

\subsection{Proof of Theorem \ref{thm:main} (ii)} \label{ssc:1.2-ii}
Let us turn to the proof of Theorem \ref{thm:main} (ii).
Throughout this subsection, we assume that $k$ is odd and greater than or equal to $3$. 
Proposition \ref{prop:bound_E} and Theorem \ref{thm:dim_LSh2} give the desired upper bound of $\dim_\Q \mathcal{DE}_k^S$, so all we should do is to verify that it also gives the lower bound.

We start with an expression of $G^{\shuffle,S}_{r,s}$ when $r+s$ is odd.

\begin{proposition}\label{prop:Gsym_dep2_odd}
Let $k\ge3$ be a positive odd integer. For $r>s\ge1$ with $k=r+s$, we obtain
\begin{equation*}\label{eq:dep2_odd}
\begin{aligned}
G^{\shuffle,S}_{r,s} &=\begin{cases} 2G_{r,s}^\shuffle + G_{k}^\shuffle & \text{\rm if } s:{\rm even},\\-2G_{s,r}^\shuffle(\tau)-G_{k}^\shuffle +\delta_{s,1}\dfrac{(2\pi i)^2}{2(k-2)}  {q\dfrac{d}{dq}G_{k-2}^\shuffle}& \text{\rm if }s:{\rm odd}.\end{cases}
\end{aligned}
\end{equation*}
\end{proposition}

\begin{proof}
By definition and the stuffle relation, for $r,s\ge2$ with $r+s$ odd, we have
\begin{align*} 
G_{r,s}^{\shuffle,S} & = G_{r,s}^\shuffle + (-1)^s G_r^\shuffle G_s^{\shuffle}-G_{s,r}^\shuffle \\
&=(1+(-1)^{s}) G_{r,s}^\shuffle + (-1)^s G_{r+s}^\shuffle  +(1-(-1)^s) G_{s,r}^\shuffle ,
\end{align*}
from which the desired equalities follow.
As in the proof of Proposition \ref{prop:Gsym_dep2_even}, the case where $s=1$ is obtained from the shuffle relation and Kaneko's sum formula \eqref{eq:sum_formula}.
\end{proof}

Note that $G_{r,s}^{\shuffle,S}=-G_{s,r}^{\shuffle,S}$ holds for $r+s$ odd because of the reversal relation.
Below, we use Proposition \ref{prop:Gsym_dep2_odd} for the case when $s$ is odd. 
To obtain the lower bound of $\dim_\Q \mathcal{DE}_k^S$, we focus on the imaginary part of $\widetilde{G}_{r,s}^{\shuffle,S}$.
Here, for a sequence $(c_n)_n\subset \C$, the $q$-series $i\sum_{n\ge0} {\rm Im} (c_n)q^n\in i\R\llbracket q\rrbracket $ is called the \emph{imaginary part} of the $q$-series $\sum_{n\ge0} c_n q^n\in \C\llbracket q\rrbracket $.
For example,  the imaginary part of $\tilde{g}_\bk^\shuffle$ equals 0 for every index $\bk$ because $\tilde{g}_\bk^\shuffle\in \Q\llbracket q\rrbracket $. We will use this fact later.
We set $\tilde{\zeta}^{\shuffle,S}(\bk)=(2\pi i)^{-\wt (\bk)} \zeta^{\shuffle,S}(\bk)$ and $\tilde{\zeta}^{\shuffle}(\bk)=(2\pi i)^{-\wt (\bk)} \zeta^{\shuffle}(\bk)$.

For $k\ge3$, let $K=\frac{k-1}{2}$.
For $\ell=1,2,\ldots,K$, let $I_{k-2\ell+1,2\ell-1}$ denote the imaginary part of $\widetilde{G}_{k-2\ell+1,2\ell-1}^{\shuffle,S}-\tilde{\zeta}^{\shuffle,S}(k-2\ell+1,2\ell-1) $ whose constant term is depleted.
From Proposition \ref{prop:Gsym_dep2_odd} and the formula \eqref{eq:q-exp_double}, we have 
\begin{align} \label{eq:I}
I_{k-2\ell+1,2\ell-1}= 2 \sum_{1\le m\le K-1} b_k(\ell,m) \tilde{\zeta}^\shuffle(2m+1)\tilde{g}^{\shuffle}_{k-2m-1},
\end{align}
where
\[ b_k(\ell,m)=\binom{2m}{2\ell-2}+\binom{2m}{k-2\ell}-\delta_{2\ell-1,2m+1} .\]
Indeed, since $\zeta^\shuffle(k)$ is a real number, one readily checks that $\tilde{\zeta}^\shuffle (k)$ is purely imaginary when $k$ is odd and real otherwise.
In addition, $b_k(\ell,m)$ is an integer and $\tilde{g}^\shuffle_{\boldsymbol{k}}\in \Q\llbracket q\rrbracket $ for any index $\boldsymbol{k}$, and thus only the terms corresponding odd $p$ in \eqref{eq:q-exp_double} survive as the imaginary part, as in \eqref{eq:I}. 

Let us now consider the coefficient matrix consisting of $b_k(\ell,m)$.
For $k\ge5$ odd, we define the $K\times (K-1)$ matrix $\mathcal{C}_k$ by 
\[ \mathcal{C}_k :=  \bigl( b_k(\ell,m) \bigr)_{\begin{subarray}{c}  1\le \ell \le K \\ 1\le m\le K-1\end{subarray}},\]
where $K=\frac{k-1}{2}$.
Here are examples of the matrices $\mathcal{C}_k$.
\begin{align*}
\mathcal{C}_5&=\begin{pmatrix}
       1 \\
       2
      \end{pmatrix}, \; 
\mathcal{C}_7=\begin{pmatrix}  1 & 1  \\
 0 & 10  \\
 2 & 4  \end{pmatrix},\; 
 \mathcal{C}_{9}=\begin{pmatrix}   1 & 1 & 1  \\
 0 & 6 & 21  \\
 0 & 4 & 35  \\
 2 & 4 & 6   \end{pmatrix},\;
\mathcal{C}_{11}=\begin{pmatrix} 1 & 1 & 1 & 1  \\
 0 & 6 & 15 & 36  \\
 0 & 0 & 21 & 126  \\
 0 & 4 & 20 & 84  \\
 2 & 4 & 6 & 8  \end{pmatrix},\\
\mathcal{C}_{13}&=\begin{pmatrix}  1 & 1 & 1 & 1 & 1  \\
 0 & 6 & 15 & 28 & 55  \\
 0 & 0 & 15 & 78 & 330  \\
 0 & 0 & 6 & 84 & 462  \\
 0 & 4 & 20 & 56 & 165  \\
 2 & 4 & 6 & 8 & 10   \end{pmatrix},\quad 
 \mathcal{C}_{15}=\begin{pmatrix} 1 & 1 & 1 & 1 & 1 & 1 \\
 0 & 6 & 15 & 28 & 45 & 78  \\
 0 & 0 & 15 & 70 & 220 & 715 \\
 0 & 0 & 0 & 36 & 330 & 1716 \\
 0 & 0 & 6 & 56 & 297 & 1287 \\
 0 & 4 & 20 & 56 & 120 & 286 \\
 2 & 4 & 6 & 8 & 10 & 12   \end{pmatrix}.
\end{align*}

The following lemma is a key ingredient of the proof of Theorem \ref{thm:main} (ii), and its proof will be given in Appendix~\ref{appendix:rank_estimate}.

\begin{lemma}\label{lem:rank_C_k}
For $k\ge5$ odd, let $\kappa=\left\lfloor\frac{k}{3}\right\rfloor$ and $K=\frac{k-1}{2}$.
For $\ell=1,2,\ldots,\kappa$, set
\begin{align} \label{eq:n_ell}
 n_\ell=
\begin{cases}
 \ell & \text{if } 1\leq \ell \leq \left\lfloor\frac{\kappa+1}{2}\right\rfloor, \\
\ell+ K-\kappa & \text{if } \left\lfloor\frac{\kappa+1}{2}\right\rfloor+1\leq \ell\leq \kappa.
\end{cases}
\end{align}
Then the submatrix
\[ \mathcal{S}_k :=\bigl(b_k(n_\ell,m)\bigr)_{\begin{subarray}{c}  1\le \ell \le \kappa \\ 1\le m\le \kappa\end{subarray}}\]
of the matrix $\mathcal{C}_k$ has full rank.
Hence, we have ${\rm rank}\, \mathcal{C}_k\ge \kappa$.
\end{lemma}

Examples of the square matrices $\mathcal{S}_k$ of degree $\kappa$ are given as follows.
\begin{align*}
\mathcal{S}_5&=\begin{pmatrix}
       1 
      \end{pmatrix}, \; 
\mathcal{S}_7=\begin{pmatrix}  1 & 1 \\
 2 & 4 \end{pmatrix},\; 
 \mathcal{S}_{9}=\begin{pmatrix}  1 & 1 & 1 \\
 0 & 6 & 21 \\
 2 & 4 & 6   \end{pmatrix},\;
\mathcal{S}_{11}=\begin{pmatrix}  1 & 1 & 1 \\
 0 & 6 & 15 \\
 2 & 4 & 6  \end{pmatrix},\\
\mathcal{S}_{13}&=\begin{pmatrix}  1 & 1 & 1 & 1 \\
 0 & 6 & 15 & 28 \\
 0 & 4 & 20 & 56 \\
 2 & 4 & 6 & 8 \end{pmatrix},\quad 
 \mathcal{S}_{15}=\begin{pmatrix}  1 & 1 & 1 & 1 & 1 \\
 0 & 6 & 15 & 28 & 45 \\
 0 & 0 & 15 & 70 & 220 \\
 0 & 4 & 20 & 56 & 120 \\
 2 & 4 & 6 & 8 & 10 \end{pmatrix}.
\end{align*}

We are ready to prove Theorem $\ref{thm:main}$ {\protect\rm (ii)}.

\begin{proof}[Proof of Theorem $\ref{thm:main}$ {\protect\rm (ii)}]
The statement for the case $k=3$ is immediate.
Let $k\ge 5$ odd.
We first prove that the set  $X_k=\{ I_{k-2n_\ell+1,2n_\ell-1}\mid \ell=1,2,\ldots,\kappa\}$ is linearly independent over $\Q$, where $n_\ell$ and $\kappa$ are given in Lemma \ref{lem:rank_C_k}.

Consider a $\Q$-linear relation $\sum_{\ell=1}^{\kappa} a_\ell I_{k-2n_\ell+1,2n_\ell-1}=0$. 
From \eqref{eq:I}, the $\C$-linear independence of $\{\tilde{g}^{\shuffle}_{k-2m-1}\mid m=1,2,\ldots,K\}$, which can be proved in much the same way as \cite[Theorem 5]{KT} by the method of reduction to a Vandermonde matrix, implies 
\begin{align*}
 2\tilde{\zeta}^{\shuffle}(2m+1)\sum_{\ell=1}^{\kappa} a_\ell b_k(n_\ell,m)=0 \quad \text{for }m=1,2,\ldots, K.
\end{align*} 
Thus, it follows from Lemma \ref{lem:rank_C_k} that the set $X_k$ is linearly independent over $\Q$.
Together with the inequality $\dim_{\mathbb{Q}} \mathcal{DE}_k^S \le \kappa$ obtained from Proposition \ref{prop:bound_E} and Theorem \ref{thm:dim_LSh2}, this shows that the set $\{ \widetilde{G}_{k-2n_\ell+1,2n_\ell-1}^{\shuffle,S}\mid \ell=1,2,\ldots,\kappa\}$ forms a basis of $ \mathcal{DE}_k^S$.

For $\ell=\left\lfloor\frac{\kappa+1}{2}\right\rfloor+1,\left\lfloor\frac{\kappa+1}{2}\right\rfloor+2,\ldots,\kappa$, by \eqref{eq:n_ell} we have $\widetilde{G}_{k-2n_\ell+1,2n_\ell-1}^{\shuffle,S}= \widetilde{G}_{2\kappa-2\ell+2,k-2\kappa+2\ell -2}^{\shuffle,S}$.
From this, we see that
\[ \{ \widetilde{G}_{k-2n_\ell+1,2n_\ell-1}^{\shuffle,S}\mid \ell=\lfloor(\kappa+1)/2\rfloor+1,\ldots,\kappa\}= \{\widetilde{G}_{2j,k-2j}^{\shuffle,S}\mid j=1,2,\ldots,\lfloor k/6\rfloor\},\]
where we have used the equality $\kappa-\lfloor(\kappa+1)/2\rfloor=\lfloor k/6\rfloor$ for simplicity.
On the other hand, for $\ell= 1,2,\ldots,\left\lfloor\frac{\kappa+1}{2}\right\rfloor$, we obtain $\widetilde{G}_{k-2n_\ell+1,2n_\ell-1}^{\shuffle,S}=- \widetilde{G}_{2\ell-1,k-2\ell+1}^{\shuffle,S}$ from the reversal relation and \eqref{eq:n_ell}.
Therefore, noting that
\[ \{2j\mid j=1,2,\ldots,\lfloor k/6\rfloor\}\cup \{2\ell-1\mid \ell= 1,2,\ldots,\lfloor(\kappa+1)/2\rfloor\}=\{1,2,\ldots,\kappa\},\]
we see that the set $\big\{\widetilde{G}_{j,k-j}^{\shuffle,S} \mid  j=1,2,\ldots, \left\lfloor \frac{k}{3} \right\rfloor \big\}$ forms a basis of $\mathcal{DE}_k^S$.
This completes the proof.
\end{proof}

\subsection{Proof of Theorem \ref{thm:modular_relations}}
\label{ssc:1.1}


As observed by Kaneko and Zagier (see also Remark~\ref{rem:KZ}), there exist linear relations among finite (or symmetric) \emph{triple} zeta values.
As a first step toward understanding this phenomenon, we prove Theorem \ref{thm:modular_relations}.

Let $\mathcal{TE}_k^{S}$ denote the $\Q$-vector space spanned by $\widetilde{G}_{\bk}^{\shuffle,S}$ of weight $k$ and depth $3$.
We have $(2\pi i)^k \mathcal{TE}_k^S = \mathcal{E}_k^{S,(3)}$.
The following lemma is a key observation of our proof of Theorem \ref{thm:modular_relations}.

\begin{lemma}\label{lem:triple}
Suppose that $k$ is even and greater than or equal to $10$. 
\begin{enumerate}[leftmargin=23pt, label={\rm (\roman*)}]
\item For each $j=1,2,\ldots, \left\lfloor\frac{k}{4}\right\rfloor$, we have $\widetilde{G}_{2j}^\shuffle \widetilde{G}_{k-2j}^\shuffle\equiv \widetilde{G}_k^\shuffle \mod \mathcal{TE}_k^{S}$.
\item The $q$-series $\widetilde{G}_k^\shuffle$ and $q\frac{d}{dq}\widetilde{G}_{k-2}^\shuffle$ lies in $\mathcal{TE}_k^{S}$.
\end{enumerate}
\end{lemma}

\begin{proof}
Let $r,s,t$ be positive integers greater than or equal to $2$ with $k=r+s+t$.
Using the stuffle relation, we obtain
\begin{equation}\label{eq:cong_GG}
\widetilde{G}^{\shuffle,S}_r\widetilde{G}^{\shuffle,S}_{s,t} 
\equiv \widetilde{G}^{\shuffle,S}_{r+s,t}+\widetilde{G}^{\shuffle,S}_{s,r+t} \mod \mathcal{TE}_k^{S}.
\end{equation}
If $r$ and $s$ are odd and greater than or equal to $3$, we have $\widetilde{G}^{\shuffle,S}_r=0$ due to \eqref{eq:dep1}, and thus we obtain
\begin{align*}
0&=\widetilde{G}^{\shuffle,S}_r\widetilde{G}^{\shuffle,S}_{s,t}\equiv \widetilde{G}^{\shuffle,S}_{r+s,t}+\widetilde{G}^{\shuffle,S}_{s,r+t} \mod\mathcal{TE}_k^{S}\\
&=2\widetilde{G}_{r+s}^\shuffle \widetilde{G}_t^\shuffle -2\widetilde{G}_k^\shuffle.
\end{align*}
Here the last equality follows from Proposition \ref{prop:Gsym_dep2_even}. This proves (i) because either $2j$ or $k-2j$ is greater than or equal to $6$; note that $10=6+4$.

Now suppose that all of $r,s$ and $t$ are even positive integers.
By definition we have
\[\widetilde{G}^{\shuffle,S}_r\widetilde{G}^{\shuffle,S}_{s,t}= 2\widetilde{G}_r^\shuffle (2\widetilde{G}_s^\shuffle \widetilde{G}_t^\shuffle -\widetilde{G}_{s+t}^\shuffle).\]
On the other hands, Proposition \ref{prop:Gsym_dep2_even} implies
\[ \widetilde{G}^{\shuffle,S}_{r+s,t}+\widetilde{G}^{\shuffle,S}_{s,r+t}=2\widetilde{G}_{r+s}^\shuffle \widetilde{G}_t^\shuffle +2\widetilde{G}_s^\shuffle \widetilde{G}_{r+t}^\shuffle -2\widetilde{G}_k^\shuffle.\]
Combining these equalities with \eqref{eq:cong_GG}, we obtain
\[ 2\widetilde{G}_r^\shuffle \widetilde{G}_s^\shuffle \widetilde{G}_t^\shuffle\equiv \widetilde{G}_r^\shuffle \widetilde{G}_{s+t}^\shuffle+\widetilde{G}_{r+s}^\shuffle \widetilde{G}_t^\shuffle +\widetilde{G}_s^\shuffle \widetilde{G}_{r+t}^\shuffle -\widetilde{G}_k^\shuffle \mod\mathcal{TE}_k^{S}.\]
Specializing at $r=s=4$ and using $(\widetilde{G}_4^\shuffle)^2=\frac{7}{6}\widetilde{G}_8^\shuffle$, we get
\[ \frac{7}{3}\widetilde{G}_8^\shuffle \widetilde{G}_{k-8}^\shuffle \equiv 2\widetilde{G}_4^\shuffle \widetilde{G}_{k-4}^\shuffle  + \widetilde{G}_8^\shuffle \widetilde{G}_{k-8}^\shuffle - \widetilde{G}_k^\shuffle \mod \mathcal{TE}_k^{S}.\]
Using (i), we obtain $\frac{1}{3}\widetilde{G}_k^\shuffle \equiv 0 \mod \mathcal{TE}^S_k$, which implies (ii) for $\widetilde{G}_k^\shuffle$.
For the derivative $q\frac{d}{dq}\widetilde{G}_k^\shuffle$, recall that the Ramanujan--Serre derivative $\partial f(\tau)=f'(\tau) -\frac{k}{12}E_2(\tau)f(\tau)$, where $E_2(\tau)=\frac{6}{\pi^2}G_2(\tau)$, sends a modular form of weight $k$ to a modular form of weight $k+2$.
This implies that there exists a modular form $g(\tau)$ of weight $k$ such that $\widetilde{G}_{k-2}'(\tau)$ is contained in $\Q g(\tau) + \Q \widetilde{G}_{k-2}(\tau)\widetilde{G}_2(\tau)$.
Hence $q\frac{d}{dq}\widetilde{G}_{k-2}^\shuffle\in \mathcal{TE}_k^{S}$ holds by Theorem~\ref{thm:two_product} combined with the statements (i) and (ii) for $\widetilde{G}^{\shuffle}_k$, as desired.
\end{proof}

We are now ready to prove Theorem $\ref{thm:modular_relations}$.

\begin{proof}[Proof of Theorem $\ref{thm:modular_relations}$]
For $k\ge10$ even, Lemma \ref{lem:triple} (i) implies that $\widetilde{G}_{2j}^\shuffle \widetilde{G}_{k-2j}^\shuffle\in \mathcal{TE}_k^S$.
Then, by Theorem \ref{thm:two_product} and Lemma \ref{lem:triple} (ii), we see that $M_k^\Q({\rm SL}_2(\Z)) \oplus \Q  \widetilde{G}_{k-2}'(\tau)$ is contained in $\mathcal{TE}_k^S$,
from which the result follows.
\end{proof}


%

\appendix
\section{The Goncharov coproduct and the Ihara group law}\label{appendix:Ihara-Goncharov}
In this appendix, We check the formula in Theorem \ref{thm:Gon_Ih} by a direct computation.
First, we write down the coefficients of all words that may appear on the left-hand side of the equation in Theorem \ref{thm:Gon_Ih}.
Then, we recall the Goncharov coproduct on the ring of formal iterated integrals and show that the terms appearing there coincide with the coefficients obtained above.

Let $A,B,\alpha$ and $\beta$ be as in Theorem \ref{thm:Gon_Ih}.
Set $\gamma(w)$ for $w\{e_0,e_1\}^\times$ to be the coefficient of $w^\vee$ in $B(X_0,AX_1A^{-1})$.
By definition, for a word $w\in \{e_0,e_1\}^\times$, we have
\[ \langle A\circ B\mid w \rangle = \sum_{w=w'w''}\gamma(w')\alpha(w'').\]
To compute $\gamma(w')$, we begin by expanding $B(X_0, AX_1A^{-1})$ with respect to the depth.
By \eqref{eq:inverse_cof}, we have $\langle A^{-1}\mid w\rangle=\alpha(\epsilon(w))$, where $\epsilon$ is defined as $\epsilon(e_{a_0}\cdots e_{a_n})=(-1)^ne_{a_n}\cdots e_{a_0}$.
Hence, it follows that 
\[AX_1A^{-1}=\sum_{w_1,w_2\in\{e_0,e_1\}^\times } \alpha(w_1)\alpha(\epsilon(w_2))w_1^\vee X_1w_2^\vee.\]
Using this, we may compute as
\begin{align*}
B(X_0, AX_1A^{-1})&=\sum_{\substack{r\ge0\\n_0,n_1\ldots,n_r\ge0}}\beta(e_0^{n_0}e_1e_0^{n_1}\cdots e_1e_0^{n_r})X_0^{n_0} \big(AX_1A^{-1}\big) X_0^{n_1}\cdots \big(AX_1A^{-1}\big)X_0^{n_r}\\
&=\sum_{\substack{r\ge0\\n_0,n_1\ldots,n_r\ge0}}\beta(e_0^{n_0}e_1e_0^{n_1}\cdots e_1e_0^{n_r})\sum_{\substack{w_{11},\ldots,w_{1r}\in\{e_0,e_1\}^\times \\w_{21},\ldots,w_{2r}\in\{e_0,e_1\}^\times}} \prod_{j=1}^r \alpha(w_{1j})\alpha(\epsilon(w_{2j}))\\
&\qquad \times X_0^{n_0} w_{11}^\vee X_1 w_{21}^\vee X_0^{n_1} \cdots w_{1r}^\vee X_1 w_{2r}^\vee X_0^{n_r}.
\end{align*}
Let us find all contributions to the coefficient of $(w')^\vee$ in the above expansion.
For instance, if $w'$ does not contain $e_1$ (i.e., $\deg_{e_1}(w')=0$), then we have $\gamma(w')=\beta(w')$, and there is no contribution from the coefficient of $AX_1A^{-1}$.
In general, suppose that $w'$ contains $e_1$ exactly $d$ times.
Then, for each $k=1,2,\ldots,d$, the contribution to $\gamma(w')$ is given by
\begin{equation}\label{eq:AppB}
\sum_{w'= u_0 e_1 u_1 e_1 u_2 \cdots e_1 u_k}
\sum_{\substack{
u_0=v_0'v_0''\\
u_1=v_1 v_1' v_1''\\
\vdots\\
u_{k-1}=v_{k-1} v_{k-1}' v_{k-1}''\\
u_k=v_k v_k'
}}
\beta(v_0' e_1 v_1' e_1 v_2' \cdots e_1 v_k')
\prod_{j=1}^k \alpha( v_{j-1}'' )\alpha(\epsilon( v_j)),
\end{equation}
where the first sum is taken over all deconcatenations of the form $w'= u_0 e_1 u_1 e_1 u_2 \cdots e_1 u_k$ (each word $u_j$ may contains $e_1$).
Summing these terms over all $k$ and adding $\beta(w')$ as the case $k=0$ in \eqref{eq:AppB} yields $\gamma(w')$.

For the right-hand side of the equation in Theorem \ref{thm:Gon_Ih}, let $\mathcal{I}$ be the ring of formal iterated integrals $\mathbb{I}(a_0;a_1,\ldots,a_n;a_{n+1})\ (a_j\in\{0,1\})$.
The $\Q$-algebra $\mathcal{I}$ is identified with $\mathfrak{H}_\shuffle$ via the map sending $\mathbb{I}(0;a_1,\ldots,a_n;1)$ to $e_{a_1}\cdots e_{a_n}$.
Note that the formal iterated integral satisfies $\mathbb{I}(1;a_1,\ldots,a_n;0)=(-1)^n \mathbb{I}(0;a_n,\ldots,a_1;1)$.
For $\varphi\in\{\alpha,\beta\}$, we use the same symbol $\varphi$ for the $\Q$-linear map $\varphi\colon \mathcal{I}\rightarrow R$ sending $\mathbb{I}(0;a_1,\ldots,a_n;1)$ to $\varphi(e_{a_1}\cdots e_{a_n})$.
Then, we have $\varphi(\mathbb{I}(1;a_1,\ldots,a_n;0))=\varphi(\epsilon(e_{a_1}\cdots e_{a_n}))$ by construction.

We now recall the explicit formula for the Goncharov coproduct $\Delta_G$:
\begin{equation}\label{eq:GonCop}
\begin{aligned}
& \Delta_G \big(\mathbb{I}(a_0;a_1,\ldots,a_n;a_{n+1})\big) \\
&\quad  =\sum_{\substack{0\le k \le n\\i_0=0<i_1<\cdots<i_k<i_{k+1}=n+1}} \prod_{p=0}^k \mathbb{I}(a_{i_p};a_{i_p+1},\ldots,a_{i_{p+1}-1};a_{i_{p+1}})  \otimes \mathbb{I}(0;a_{i_1},\ldots,a_{i_k};1).
\end{aligned}
\end{equation}
Each term in the sum is determined by a choice of positions $i_1,\ldots,i_k$ with 
$i_1<\cdots<i_k$.
We represent this choice by underlining the letters corresponding to
$i_1,\ldots,i_k$ and write
\[
I(e_{a_1}\cdots \underline{e_{a_{i_1}}}\cdots \underline{e_{a_{i_2}}}\cdots \underline{e_{a_{i_k}}}\cdots e_{a_n}) = \prod_{p=0}^k \mathbb{I}(a_{i_p};a_{i_p+1},\ldots,a_{i_{p+1}-1};a_{i_{p+1}})  \otimes \mathbb{I}(0;a_{i_1},\ldots,a_{i_k};1).\]
Then, for a word $w\in \{e_0,e_1\}^\times$, the sum defining $\Delta_G(w)$ is taken over
all possible ways of placing underlines on the letters of $w$.
To compute this, we first consider all deconcatenations of the form $w=w'w''$, and then all possible ways of placing underlines on the letters of $w'$.
Suppose that $w'$ contains $e_1$ exactly $d$ times.
For each $k = 1,2,\ldots,d$, we further consider all deconcatenations of the form $w'=u_0e_1u_1\cdots e_1u_k$, and, for each $j$, deconcatenations of the form $u_j=v_jv_j'v_j''$ with $v_0=v_k''=e_\varnothing$.
For this deconcatenation of $w'$, we consider the following way of placing underlines:
\begin{equation}\label{eq:II}
I(\underline{v_0'}v_0'' \underline{e_1} v_1 \underline{v_1'} v_1'' \underline{e}_1v_2 \underline{v_2'}\cdots \underline{v_{k-1}'}v_{k-1}'' \underline{e_1}v_k \underline{v_k'} w'') .
\end{equation}
Let $v_{j,L}'$ (resp.~$v_{j,R}'$) be the leftmost (resp.~rightmost) letter of $v_j'$.
Then, the above term can be written as
\begin{align*}
\left\{\prod_{j=1}^k \mathbb{I}(v_{j-1,R}';v_{j-1}'';1)\mathbb{I} (1;v_j;v_{j,L}') \right\} \mathbb{I}(v_{k,R}';w';1)\otimes \mathbb{I}(0;v_0'e_1v_1'e_1v_2'\cdots e_1 v_k';1).
\end{align*}
Note that, if $v_j''$ (resp.~$v_j$) is nonempty and $v_{j,R}' = e_1$ (resp.~$v_{j,L}'=e_1$), then $\mathbb{I}(v_{j,R}'; v_j''; 1)$ (resp.~$\mathbb{I} (1;v_j;v_{j,L}')$) vanishes due to \cite[Definition 3.2 (I4)]{BT}.
If $v_{j,R}' = e_0$ (resp.~$v_{j,L}'=e_0$), we have
\[
\alpha\bigl(\mathbb{I}(v_{j,R}'; v_j''; 1)\bigr) = \alpha(v_j'') \quad ({\rm resp}.~\alpha\bigl(\mathbb{I}(1;v_j;v_{j,L}')\bigr)= \alpha(\epsilon(v_j)) ).
\]
Hence, applying $m\circ (\alpha\otimes \beta)$ to \eqref{eq:II}, we get $\eqref{eq:AppB} \times \alpha(w'')$.
We leave it to the reader to verify that no other terms contribute to $\Delta_G(w)$.

\section{Rank estimation of a binomial matrix} \label{appendix:rank_estimate}

We will prove Lemma \ref{lem:rank_C_k} in the following.
Let $k$ be an odd integer greater than or equal to $5$. 
Note that the matrix $\mathcal{S}_k$ has the following form.
\begin{align*}
 \mathcal{S}_k &=
\begin{pmatrix}
 \displaystyle \binom{2}{0} &  \displaystyle \binom{4}{0} & \displaystyle  \binom{6}{0} &  \cdots & \cdots & \cdots & \displaystyle \binom{2\kappa}{0} \\[1em]
0 &  \displaystyle \binom{4}{2} & \displaystyle \binom{6}{2} & \cdots & \cdots & \cdots & \displaystyle  \binom{2\kappa}{2} \\[.7em]   
 \vdots & \ddots & \ddots & \ddots & \ddots & \ddots & \vdots \\[.7em]
 0 & 0 &  \cdots & 0 & \displaystyle \binom{\kappa+\frac{1-(-1)^\kappa}{2}}{\kappa-1} & \cdots  &  \displaystyle \binom{2\kappa}{\kappa-1} +\delta_{3\mid k} \binom{2\kappa}{2\kappa-1} \\[1em] 
 \vdots & \iddots & \iddots & \iddots & \iddots & \iddots &  \vdots \\[.7em]
0 &  \displaystyle \binom{4}{3} & \displaystyle \binom{6}{3} & \cdots & \cdots & \cdots & \displaystyle  \binom{2\kappa}{3} \\[1em]
 \displaystyle \binom{2}{1} &  \displaystyle \binom{4}{1} & \displaystyle  \binom{6}{0} &  \cdots & \cdots & \cdots & \displaystyle \binom{2\kappa}{1}
\end{pmatrix},
\end{align*}
where $\delta_{3\mid k}$ denotes $1$ when $k$ is divisible by $3$ and $0$ otherwise. 

To verify Lemma \ref{lem:rank_C_k}, we shall perform Gaussian elimination on $\mathcal{S}_k$ in a slightly elaborate way to reduce $\mathcal{S}_k$ into an upper triangular matrix with nonzero diagonal entries. The next binomial identity, which seems to be of independent interest, plays a crucial role in the reduction procedure.
Hereafter we set $\displaystyle \binom{n}{m}=0$ unless the integers $n$ and $m$ satisfy $n\geq m\geq 0$.

\begin{proposition} \label{prop:binom_id}
 For any non-negative integers $j$ and $a$, we have 
\begin{align} \label{eq:binom_id}
 \binom{j}{a}-(-1)^j\delta_{j,a}=\sum_{\ell=\left\lfloor\frac{a}{2}\right\rfloor}^{a} \left\{2 \binom{\ell+1}{a-\ell}-\binom{\ell}{a-\ell}\right\}\binom{j-\ell-1}{\ell}
\end{align}
where $\delta_{j,a}$ denotes the Kronecker delta.
\end{proposition}

\begin{proof}
Let $\phi_{j,a}$ denote the right-hand side of \eqref{eq:binom_id}.
 We determine the generating function of the sequence $\bigl(\phi_{j,a}\bigr)_{j,a}$. We first calculate $\sum_{a=0}^j \phi_{j,a}X^a$ by interchanging the order of summation as
\begin{align*}
 \sum_{a=0}^j \phi_{j,a}X^a &=\sum_{\ell=0}^{\left\lfloor\frac{j-1}{2}\right\rfloor}\left[ 
\sum_{a=\ell}^{2\ell+1} \left\{2\binom{\ell+1}{a-\ell}-\binom{\ell}{a-\ell}\right\}X^a \right]\binom{j-\ell-1}{\ell} \\
&\qquad +\sum_{\ell=\left\lfloor\frac{j+1}{2}\right\rfloor}^{j}\left[ 
\sum_{a=\ell}^{j} \left\{2\binom{\ell+1}{a-\ell}-\binom{\ell}{a-\ell}\right\}X^a \right]\binom{j-\ell-1}{\ell},
\end{align*}
but the second sum vanishes since the binomial coefficient $\binom{j-\ell-1}{\ell}$ equals $0$ for each $\ell=\left\lfloor\frac{j+1}{2}\right\rfloor, \left\lfloor\frac{j+1}{2}\right\rfloor+1,\ldots,j$. Therefore we obtain
\begin{align*}
\sum_{j=0}^\infty \sum_{a=0}^j\phi_{j,a}X^a Y^j &=\sum_{j=0}^\infty \left[\sum_{\ell=0}^{\left\lfloor\frac{j-1}{2}\right\rfloor} \{2X^\ell (1+X)^{\ell+1}-X^\ell (1+X)^\ell\} \binom{j-\ell-1}{\ell}\right] Y^j \\
&=(2X+1)\sum_{j=0}^\infty\sum_{\ell=0}^{\left\lfloor\frac{j-1}{2}\right\rfloor} \binom{j-\ell-1}{\ell} X^\ell (1+X)^\ell Y^j \\
&=\dfrac{(2X+1)Y}{1-X(X+1)-X(X+1)Y^2} =\dfrac{1}{1-(1+X)Y}-\dfrac{1}{1+XY}.
\end{align*}
Here we use a well-known formula 
\begin{align*}
 \sum_{j=0}^\infty \sum_{\ell=0}^{\left\lfloor\frac{j-1}{2}\right\rfloor} \binom{j-\ell-1}{\ell} z^\ell Y^j=\dfrac{Y}{1-z-zY}
\end{align*}
at the second equality. 
Meanwhile, the generating function of  $\{ \binom{j}{a} -(-1)^j \delta_{j,a}\}_{j,a}$ is readily calculated as 
\begin{align*}
 \sum_{j=0}^\infty \left[\sum_{a=0}^j \left\{\binom{j}{a} -(-1)^j \delta_{j,a}\right\} X^a\right] Y^j&=\sum_{j=0}^\infty \left\{ (1+X)^j -(-X)^j\right\}Y^j \\
&=\dfrac{1}{1-(1+X)Y}-\dfrac{1}{1+XY}.
\end{align*}
Since the generating functions coincide with each other, we obtain an identity \eqref{eq:binom_id} as desired.
\end{proof}

\begin{corollary}
 For positive integers $\ell'$ and $m$, we have
\begin{align} \label{eq:binom_id2}
\binom{2m}{\ell'-1}-\delta_{\ell'-1,2m} =\binom{2m-\ell'}{\ell'-1} \! +\sum_{\nu=1}^{\left\lfloor\frac{\ell'}{2}\right\rfloor} \left\{ \binom{\ell'-\nu}{\nu}+\binom{\ell'-\nu-1}{\nu-1}\right\} \! \binom{2m-\ell'+\nu}{\ell'-\nu-1}.
\end{align}
\end{corollary}

\begin{proof}
 Using Pascal's relation $\binom{\ell}{a-\ell-1}+\binom{\ell}{a-\ell}=\binom{\ell+1}{a-\ell}$, we can rewrite \eqref{eq:binom_id} as
\begin{align*}
 \binom{j}{a}-(-1)^j\delta_{j,\ell'}=\sum_{\ell=\left\lfloor\frac{\ell'}{2}\right\rfloor}^{\ell'}\left\{ \binom{\ell+1}{a-\ell}+\binom{\ell}{a-\ell-1}\right\} \binom{j-\ell-1}{\ell}.
\end{align*}
Replacing $a$ by $\ell'-1$, $j$ by $2m$ and $\ell$ by $\ell'-\nu-1$, we obtain the desired equality.
\end{proof}

\begin{proof}[Proof of Lemma $\ref{lem:rank_C_k}$]
 By swapping rows appropriately, one may transform $\mathcal{S}_k$ into a square matrix $\mathcal{S}_k'+2\kappa \delta_{3\mid k}E_{\kappa,\kappa}$, where $\mathcal{S}_k'$ is defined as 
\begin{align*}
 \mathcal{S}_k'&=\left(\binom{2m}{\ell'-1}-\delta_{\ell'-1,2m}\right)_{\begin{subarray}{c}  1\le \ell' \le \kappa \\ 1\le m\le \kappa\end{subarray}}  \\
&=
\begin{pmatrix}
 \displaystyle \binom{2}{0} & \displaystyle \binom{4}{0} & \cdots & \cdots & \cdots &\displaystyle \binom{2\kappa-2}{0} & \displaystyle \binom{2\kappa}{0} \\[1em ]
 \displaystyle \binom{2}{1} &  \displaystyle\binom{4}{1} & \cdots & \cdots & \displaystyle \cdots & \displaystyle \binom{2\kappa-2}{1} &\displaystyle \binom{2\kappa}{1} \\[1em]
\displaystyle 0 & \displaystyle\binom{4}{2} & \cdots & \cdots & \cdots &\displaystyle \binom{2\kappa-2}{3} & \displaystyle\binom{2\kappa}{2} \\[1em]
\displaystyle 0 &  \displaystyle \binom{4}{3} & \cdots & \cdots & \cdots &  \displaystyle \binom{2\kappa-2}{4} &\displaystyle \binom{2\kappa}{0} \\ 
 \vdots & \vdots &\ddots & \ddots & \ddots & \vdots  & \vdots \\
 \vdots & \vdots &\ddots & \ddots & \ddots & \vdots  & \vdots \\
  0 & 0 & \cdots & \cdots & \cdots  & 0 &  \displaystyle\binom{\kappa}{\kappa-1}
\end{pmatrix}
\end{align*}
and $E_{\kappa,\kappa}$ denotes a square matrix of degree $\kappa$ with a $1$ at the $(\kappa,\kappa)$-entry and and $0$'s elsewhere.
We claim that, by performing row elementary operations (with coefficients in $\mathbb{Z}$), one may transform $\mathcal{S}_k'+2\kappa\delta_{3\mid k}E_{\kappa,\kappa}$ into an upper triangular matrix $\mathcal{T}_k+2\kappa\delta_{3\mid k}E_{\kappa,\kappa}$, where $\mathcal{T}_k$ is defined as 
\begin{align*}
 \mathcal{T}_k&=\left(\binom{2m-\ell'}{\ell'-1}\right)_{\begin{subarray}{c}  1\le \ell' \le \kappa \\ 1\le m\le \kappa\end{subarray}} =
\begin{pmatrix}
 \displaystyle \binom{1}{0} & \displaystyle \binom{3}{0} & \displaystyle \binom{5}{0} & \cdots & \cdots & \displaystyle \binom{2\kappa-1}{0} \\[1em ]
 0 &  \displaystyle\binom{2}{1} & \displaystyle \binom{4}{1} & \cdots & \displaystyle \cdots & \displaystyle \binom{2\kappa-2}{1} \\[1em]
\displaystyle 0 & 0 & \displaystyle \binom{3}{2} & \cdots & \cdots  & \displaystyle\binom{2\kappa-3}{2} \\[1em]
 \vdots & \vdots &\vdots & \ddots & \ddots & \vdots\\
 \vdots & \vdots &\vdots & \ddots & \ddots & \vdots \\
  0 & 0 & 0 & \cdots & \cdots  &   \displaystyle\binom{\kappa}{\kappa-1}
\end{pmatrix}.
\end{align*}

For any $\ell'=1,2,\ldots,\kappa$, let $\mathcal{S}_k^{(\ell')}$ denote the matrix $\mathcal{S}_k'$ whose $\nu$-th row is replaced by that of $\mathcal{T}_k$ for $\nu=1,2,\ldots,\ell'$; and thus we have $\mathcal{S}_k^{(\kappa)}=\mathcal{T}_k$. We complete the proof by induction on $\ell'$; namely we prove that $\mathcal{S}_k^{(1)}=\mathcal{S}_k'$ and, if we can reduce $\mathcal{S}_k'$ into $\mathcal{S}_k^{(\ell'-1)}$, we may transform it into $\mathcal{S}_k^{(\ell')}$ by further row elementary operations.
Indeed, the identity \eqref{eq:binom_id2} with $\ell'=1$, which is none other than the trivial identity $\binom{2m}{0}=\binom{2m-1}{0}$, implies the equality $\mathcal{S}_k^{(1)}=\mathcal{S}_k'$. Now suppose that, after several row operations, one has reduced $\mathcal{S}_k'$ into $\mathcal{S}_k^{(\ell'-1)}$ for $2\leq \ell'\leq\kappa$. The entries of $\mathcal{S}_k^{(\ell')}$ belonging to the $\nu$-th row are then given by $\left( \binom{2m-\ell'}{\ell'-1}\right)_{1\leq m\leq \kappa}$ for $\nu=1,2,\ldots,\ell'-1$ by assumption, whereas the entries belonging to the $\ell'$-th row are given by $\left( \binom{2m}{\ell'-1}-\delta_{\ell'-1,2m} \right)_{1\leq m\leq \kappa}$. Therefore if one subtracts the $(\ell'-\nu)$-th row multiplied by $\binom{\ell'-\nu}{\nu}+\binom{\ell'-\nu-1}{\nu-1}$ from the $\ell'$-th row for every $\nu=1,2,\ldots,\ell'-1$, one may transform $\mathcal{S}_k^{(\ell'-1)}$ into $\mathcal{S}_k^{(\ell')}$ due to \eqref{eq:binom_id2}, as desired.

Finally note that, throughout the whole reduction procedure from $\mathcal{S}_k'$ into $\mathcal{T}_k$, which consists of only row addition, we have never added a multiple of the $\kappa$-th row to another row. Therefore if we perform the same row operations on $2\kappa\delta_{3\mid k}E_{\kappa,\kappa}$, it remains unchanged because $E_{\kappa,\kappa}$ contains non-zero entries only in the $\kappa$-th row. We have thus verified that we may transform $\mathcal{S}_k'+2\kappa\delta_{3\mid k}E_{\kappa,\kappa}$ into $\mathcal{T}_k'+2\kappa\delta_{3\mid k}E_{\kappa,\kappa}$ by the {\em same} row elementary operations, which completes the proof.
\end{proof}

\begin{remark}
 The readers should compare the matrix $\mathcal{S}_k'=(\binom{2m}{\ell'-1}-\delta_{\ell'-1,2m})_{1\leq \ell',m\leq \kappa}$ appearing in the proof of Lemma \ref{lem:rank_C_k} with the matrix $M_\kappa=(\binom{2m-1}{\ell'-1}-\delta_{\ell'-1,2m-1})_{1\leq \ell',m\leq \kappa}$ appearing in \cite[(A.1)]{Matthes}. This similarity  also seems to  imply an underlying relationship between symmetric double Eisenstein series and elliptic double zeta values and.
\end{remark}

\section*{Acknowledgments}
This work is partially supported by
JSPS KAKENHI Grant Numbers 22K03237 (T.H.), 23K03069 (K.S.) and 23K03034 (K.T.).

\section*{Statements and Declarations}
We wish to confirm that there are no known conflicts of interest associated with this publication and there has been no significant financial support for this work that could have influenced its outcome.

The data that support the findings of this study are available from the corresponding author, K.T.\ upon reasonable request.


\end{document}